\documentclass{amsart}
\usepackage{style}
 \usepackage{mdframed} 

\begin{document}

\title[Universal Quantum Semigroupoids]
{Universal Quantum Semigroupoids}

\author{Hongdi Huang, Chelsea Walton, Elizabeth Wicks, and Robert Won}

\address{Huang: Department of Mathematics, Rice University, Houston, TX 77005, USA}
\email{h237huan@rice.edu}

\address{Walton: Department of Mathematics, Rice University, Houston, TX 77005, USA}
\email{notlaw@rice.edu}

\address{Wicks: Microsoft Corporation, Redmond, WA 98052, USA} 
\email{elizabeth.wicks@microsoft.com}

 \address{Won: Department of Mathematics, The George Washington University, Washington, DC 20052,~USA}
\email{robertwon@gwu.edu}

\begin{abstract} 
We introduce the concept of a universal quantum linear 
semigroupoid \linebreak (UQSGd), which is a weak bialgebra that coacts on a (not necessarily connected) graded algebra $A$ universally while preserving grading. We restrict our attention to algebraic structures with a commutative base so that the UQSGds under investigation are face algebras (due to Hayashi). The UQSGd construction generalizes the universal quantum linear semigroups introduced by Manin in 1988, which are bialgebras that coact on a connected graded algebra universally while preserving grading. Our main result is that when $A$ is the path algebra $\kk Q$ of a finite quiver $Q$, each of the various UQSGds introduced here is isomorphic to the face algebra attached to $Q$. The UQSGds of preprojective algebras and of other algebras attached to quivers are also investigated. 
\end{abstract}


\maketitle


\setcounter{tocdepth}{2}


\section{Introduction} \label{sec:intro}

Broadly, this work initiates the study of universal quantum symmetries in the weak context, which is needed as weak Hopf algebras naturally arise in the theory of fusion categories \cite{ENO}, in capturing symmetries in subfactor theory \cite{NV2000} (e.g., as remarked in \cite{BCG}), and in studying solutions of the dynamical Yang-Baxter equation \cite{EN}. The goal of this work is to examine the universal quantum symmetries of $\mathbb{N}$-graded algebras, that are not necessarily connected, within the framework of weak bialgebra coactions.
 All algebraic structures here are $\kk$-linear, for $\kk$ an arbitrary base field, and we reserve $\otimes$ to mean~$\otimes_{\kk}$. Consider the hypotheses below.

\begin{hypothesis}[$A, A_0$] \label{hyp:A}
Let $A$ be a locally finite, $\N$-graded $\kk$-algebra $A$, that is, $A$ has $\kk$-vector space decomposition $\bigoplus_{i \in \mathbb{N}} A_i$ with $A_i \cdot A_j  \subseteq  A_{i+j}$,  and $\dim_\kk A_i <  \infty$. 
Further suppose that the degree 0 component $A_0$, which is a finite-dimensional $\kk$-subalgebra of $A$, is a commutative and separable $\kk$-algebra. In particular, this implies that $A_0$ is a Frobenius algebra over $\kk$.
We say that $A$ is {\it connected} if $A_0 = \kk$, although we do not assume that condition here.
\end{hypothesis}


The prototypical examples of algebras $A$ whose symmetries we will examine are path algebras of finite quivers. Throughout, we fix the following notation.

\begin{notation}[$Q, Q_0, Q_1, s, t, \kk Q, e_i, p, q, a, b$] \label{not:quiver}
Let $Q = (Q_0,Q_1,s,t)$ be a finite {\it quiver} (i.e., a directed graph), where $Q_0$ is a finite collection of vertices, $Q_1$ is a finite collection of arrows, and $s, t:  Q_1 \to Q_0$ denote the source and target maps, respectively. We read paths of $Q$ from left-to-right. 
Let $\kk Q$ be the \emph{path algebra} attached to $Q$, which is the $\kk$-algebra generated by $\{e_i\}_{i \in Q_0}$ and $\{p\}_{p \in Q_1}$, with multiplication given by $m(e_i \otimes e_j) = \delta_{i,j} e_i$ for $i,j \in Q_0$ and $m(p \otimes q) = \delta_{t(p),s(q)}pq$ for $p,q \in Q_1$,  and with unit given by $u(1_\kk) = \sum_{i \in Q_0} e_i$.  The path algebra $\kk Q$ is $\N$-graded by path length, where for each $\ell \in \N$, $(\kk Q)_\ell = \kk(Q_\ell)$, where $Q_\ell$ consists of paths of length $\ell$ in $Q$. We usually use the letters $a,b$ to denote paths in $Q$.
\end{notation}

Using path algebras as the prototypical examples of not necessarily connected $\kk$-algebras $A$ in Hypothesis~\ref{hyp:A} is apt because  if $A$ is generated by $A_1$ over $A_0$, then $A$ is  isomorphic to a quotient of some path algebra $\kk Q$. Namely,  $A_0$ 
 is isomorphic to the path algebra on an arrowless quiver $Q_0$ with $|Q_0| = \dim_\kk A_0$. Further, path algebras are free structures in the sense that they are tensor algebras: $\kk Q \cong T_{\kk Q_0} (\kk Q_1)$, where $\kk Q_1$ is a $\kk Q_0$-bimodule; we will return to this fact later in the introduction. Moreover, interesting examples of graded quotients of path algebras include  preprojective algebras  \cite{G-P, Ringel}  and superpotential algebras (see, e.g., \cite{BSW}).

\smallskip

Before we study quantum symmetries of the algebras $A$ in Hypothesis~\ref{hyp:A}, let us recall various notions of a universal bialgebra coacting on $A$ in the case when $A$ is connected, which are all due to Manin \cite{Manin} (for the case when $A$ is quadratic). First, we make the standing assumption that will be used throughout this work often without mention.

\begin{hypothesis} \label{hyp:linear-intro}
Let $\mathcal{C}$ be a monoidal category of corepresentations of an algebraic structure $H$. If $A$ is an algebra in $\mathcal{C}$ (i.e., if $A$ is an $H$-comodule algebra), then we assume that each graded component $A_i$ of $A$ is an object in $\mathcal{C}$ (i.e., is an $H$-comodule). Namely, we assume that all coactions of $H$ on $A$ preserve grading, or are {\it linear}, in this work.
\end{hypothesis}

Consider the following universal bialgebras that coact on a connected algebra $A$ as in Hypothesis~\ref{hyp:A}, from either the left or right.  

\begin{definition}[left UQSG, $O^{\text{left}}(A)$; right UQSG, $O^{\text{right}}(A)$] \cite[Chapter~4 and~Sections~5.1--5.8]{Manin} \label{def:UQSG-intro}
Let $A$ be a $\kk$-algebra as in Hypothesis~\ref{hyp:A} that is connected.  
\begin{enumerate}[(a),font=\upshape]
    \item \label{itm:leftUQSG-intro} Let $O:=O^{\text{left}}(A)$ be a bialgebra for which $A$ is a left $O$-comodule algebra via  left $O$-comodule map $\lambda^{O}: A \to O \otimes A$.
We call $O^{\text{left}}(A)$ the {\it left universal quantum linear semigroup (left UQSG) of $A$} if, for any bialgebra $H$ for which $A$ is a left $H$-comodule algebra via  left $H$-comodule map $\lambda^{H}: A \to H \otimes A$, there exists a unique bialgebra map  $\pi: O \to  H$ so that $\lambda^H = (\pi \otimes \id_A) \lambda^{O}$.

\medskip

    \item \label{itm:rightUQSG-intro} Let $O:=O^{\text{right}}(A)$ be a bialgebra for which $A$ is a right $O$-comodule algebra via right $O$-comodule map $\rho^{O}: A \to A \otimes O$.
We call $O^{\text{right}}(A)$ the {\it right universal quantum linear semigroup (right UQSG) of $A$} if, for any bialgebra $H$ for which $A$ is a right $H$-comodule algebra via right $H$-comodule map $\rho^{H}: A \to A \otimes H$,  there exists a unique bialgebra map $\pi: O \to  H$ so that $\rho^H = (\id_A \otimes \pi) \rho^{O}$.
\end{enumerate}
\end{definition}

Other appearances of bialgebras that coact linearly and universally on algebraic structures from one side include the universal bi/Hopf  algebras that coact on (skew-)polynomial algebras in \cite{RR-Taft, LauveTaft, CFR}, and the universal bi/Hopf algebras that coact on a superpotential algebra (or, equivalently that preserve a certain multilinear form) in \cite{DVL, BichonDV, CWW}. 

\smallskip

Ideally, a universal bialgebra should behave ring-theoretically and homologically like the algebra that it coacts on. But this is not the case even when the algebra is a polynomial ring in two variables; see Example~\ref{ex:polynomialManin}\ref{itm:polynomial-leftright}. Namely, $\Oleft(\kk[x,y])$ is a non-Noetherian algebra of infinite Gelfand--Kirillov (GK) dimension, whereas $\kk[x,y]$ is Noetherian of GK-dimension~2. 
Towards the goal above, one can consider a `smaller' universal bialgebra introduced by Manin, which coacts an algebra $A$ universally from the left and right via `transposed' coactions. Indeed, Manin inquired if such a universal bialgebra reflects the behavior of $A$ (in the connected and quadratic case) in \cite[Introduction]{AST}. 

\begin{definition}[transposed UQSG, $O^{\textnormal{trans}}(A)$]  (cf., \cite[Section~5.10, Chapters~6 and~7]{Manin}) \label{def:ManinUQSG-intro}
Let $A$ be a $\kk$-algebra as in Hypothesis~\ref{hyp:A} that is connected.  
\begin{enumerate}[(a),font=\upshape]
    \item \label{itm:transpose-intro} Let $H$ be a bialgebra for which $A$ is a left $H$-comodule algebra via left $H$-comodule map $\lambda^{H}_A$, and for which $A$ is a right $H$-comodule algebra via  right $H$-comodule map $\rho^{H}_A$. We call $A$ a {\it transposed $H$-comodule algebra} if,  for  the transpose of $\rho^H_A$, 
    $$\quad \quad \quad (\rho_A^H)^T: (\mathrm{ev}_A \otimes \id_H \otimes \id_A)(\id_{A^*} \otimes \rho^H_A \otimes \id_{A^*})(\id_{A^*} \otimes \mathrm{coev}_A): A^* \to H \otimes A^*,$$
    we obtain $\lambda^H_A$ by identifying a basis of $A$ with the dual basis of $A^*$.

\medskip

    \item \label{itm:ManinUQSG-intro} Let $O:= O^{\textnormal{trans}}(A)$ be a bialgebra for which $A$ is a transposed $O$-comodule algebra via left $O$-comodule map $\lambda^{O}$ and right $O$-comodule map $\rho^{O}$.
We call $O^{\textnormal{trans}}(A)$ the {\it transposed universal quantum linear semigroup (transposed UQSG) of $A$} if, for any bialgebra $H$ for which $A$ is a transposed $H$-comodule algebra via left $H$-comodule map $\lambda^{H}$ and a right $H$-comodule map $\rho^{H}$, there exists a unique bialgebra map  $\pi: O \to  H$ so that $\lambda^H = (\pi \otimes \id_A) \lambda^{O}$ and $\rho^H = (\id_A \otimes \pi) \rho^{O}$.
\end{enumerate}
\end{definition}

Other instances of bialgebras that coact linearly and universally on algebraic structures in a transposed manner include the universal bi/Hopf algebras that coact on skew-polynomial algebras in \cite{Takeuchi, AST} (these are special cases of the construction in \cite{Manin}),  and the universal bi/Hopf algebras that coact on a superpotential algebras in \cite{CWW} (this is a generalization of the construction in \cite{Manin}).

\smallskip

In order to study the quantum symmetries of an algebra $A$ which satisfies Hypothesis~\ref{hyp:A}, but is not necessarily connected,
we use coactions of weak bialgebras, which are structures that have the underlying structure of an algebra and a coalgebra, with weak compatibility conditions between these substructures [Definition~\ref{def:wba}]. For a weak bialgebra $H$, there are two important coideal subalgebras, $H_s$ and $H_t$, called the source and target counital subalgebras, that measure how far $H$ is from being a bialgebra. Namely, $H$ is a bialgebra if and only if both $H_s$ and $H_t$ are the ground field $\kk$. These subalgebras are always separable and Frobenius (see Proposition~\ref{prop:Hs-facts}\ref{itm:fd-sep}).

\smallskip

Since we are considering quantum symmetries of algebras $A$ whose degree 0 components $A_0$ are commutative separable algebras, we will work within the framework of weak bialgebras with commutative counital subalgebras, which are the same as {\it $\mathcal{V}$-face algebras} by \cite[Theorem~4.3]{Schauenburg-xR} and \cite[Theorems~5.1 and~5.5]{Schauenburg-wha}. Here, $\mathcal{V}$ is a finite set. 
A key example of a $\mathcal{V}$-face algebra is the weak bialgebra $\hay$ attached to a finite quiver $Q$, which was introduced by Hayashi in \cite{Hayashi93,Hayashi96}.  In this case, $\mathcal{V} = Q_0$ and a presentation of $\hay$ is provided in  Example~\ref{ex:hay}. Next, we propose a conjecture, which is a modification of   \cite[Proposition~2.1]{Hayashi99b} that remains unproved. 

\begin{conjecture}
Suppose that  $\kk$ is algebraically closed. If $H$ is a finite-dimensional weak bialgebra with commutative counital subalgebras, then $H$ is isomorphic to a weak bialgebra quotient of $\hay$ for some finite quiver $Q$.
\end{conjecture}

This is akin to the result that every finite-dimensional algebra over an algebraically closed field is isomorphic to a  quotient of a path algebra of some finite quiver (see, e.g., \cite[Theorem~II.3.7]{ASS2006}).



\smallskip

Returning to the study of the quantum symmetries of algebras $A$ as in Hypothesis~\ref{hyp:A}, we proceed by realizing such an algebra $A$ as a comodule algebra over a weak bialgebra $H$ (which will eventually have commutative counital subalgebras). 
For a weak bialgebra $H$, let ${}^{H}\hspace{-.04in}\mathcal{A}$ (resp., $\mathcal{A}^H$) denote the category of left (resp., right) $H$-comodule algebras. An example of an object in ${}^{H}\hspace{-.04in}\mathcal{A}$ (resp., in $\mathcal{A}^H$) is $H_t$ (resp., $H_s$) via comultiplication  [Examples~\ref{ex:right} and~\ref{ex:left}]. Moreover, if an algebra $A$ satisfying Hypothesis~\ref{hyp:A} belongs to $\HA$ (resp., $\AH$), then so does the subalgebra $A_0$ [Remark~\ref{rem:A0}]. We are now ready to introduce various notions of a universal weak bialgebra coacting on $A$ which are the focus of our work.

\begin{definition}[left UQSGd, $\Oleft(A)$; right UQSGd, $\Oright(A)$;  trans. \hspace{-.1in} UQSGd, $\Otrans(A)$] \label{def:UQSGd-intro}
Let $A$ be $\kk$-algebra as in Hypothesis~\ref{hyp:A}.
\begin{enumerate}[(a),font=\upshape]
    \item \label{itm:leftUQSGd-intro} Let $\mathcal{O}:=\Oleft(A)$ be a weak bialgebra  so that $A \in {}^{\mathcal{O}}\hspace{-.04in}\mathcal{A}$  via left $\mathcal{O}$-comodule map $\lambda^{\mathcal{O}}$ with $\mathcal{O}_t \cong A_0$ in ${}^{\mathcal{O}}\hspace{-.04in}\mathcal{A}$. We call $\Oleft(A)$ the {\it left universal quantum linear semigroupoid (left UQSGd) of $A$} if, for any weak bialgebra $H$ such that $A \in \HA$ via left $H$-comodule map $\lambda^H$ with $H_t \cong A_0$ in $\HA$, there exists a unique weak bialgebra map  $\pi: \mathcal{O} \to  H$ so that $\lambda^H = (\pi \otimes \id_A) \lambda^{\mathcal{O}}$.

\medskip

    \item \label{itm:rightUQSGd-intro} Let $\mathcal{O}:=\Oright(A)$ be a weak bialgebra so that $A \in \mathcal{A}^{\mathcal{O}}$  via right $\mathcal{O}$-comodule map $\rho^{\mathcal{O}}$ with $\mathcal{O}_s \cong A_0$ in $\mathcal{A}^{\mathcal{O}}$. We call $\Oright(A)$ the {\it right universal quantum linear semigroupoid (right UQSGd) of $A$} if, for any weak bialgebra $H$ such that $A \in \AH$ via right $H$-comodule map $\rho^H$ with $H_s \cong A_0$ in $\AH$, there exists a unique weak bialgebra map  $\pi: \mathcal{O} \to  H$ so that $\rho^H = (\id_A \otimes \pi) \rho^{\mathcal{O}}$.

\medskip

\item \label{itm:ManinUQSGd-intro} Let $\mathcal{O}:=\Otrans(A)$ be a weak bialgebra so that $A \in {}^{\mathcal{O}}\hspace{-.04in}\mathcal{A}$  and $A \in \mathcal{A}^{\mathcal{O}}$   so that $A$ is a transposed $\mathcal{O}$-comodule algebra, and with  $\mathcal{O}_t \cong A_0$ in ${}^{\mathcal{O}}\hspace{-.04in}\mathcal{A}$ and  $\mathcal{O}_s \cong A_0$ in $\mathcal{A}^{\mathcal{O}}$.
We call $\Otrans(A)$ the {\it transposed universal quantum linear semigroupoid (transposed UQSGd) of $A$} if, for any weak bialgebra $H$ such that $A \in \HA$ and $A \in \AH$ for which $A$ is a transposed $H$-comodule algebra, and with  $H_t \cong A_0$ in $\HA$ and  $H_s \cong A_0$ in $\AH$, there exists a unique weak bialgebra map  $\pi: \mathcal{O} \to  H$ so that $\lambda^H = (\pi \otimes \id_A) \lambda^{\mathcal{O}}$ and $\rho^H = (\id_A \otimes \pi) \rho^{\mathcal{O}}$.
\end{enumerate}
\end{definition}

\medskip

Discussion about these definitions is provided in Remarks~\ref{rem:naive}, \ref{rem:A0fd-sep}--\ref{rem:gen-Manin-oneside}, \ref{rem:A0comm}--\ref{rem:uniqueness}; the most important observation is that, without the condition that the `base' of the weak bialgebra is isomorphic to the `base' of the comodule algebra, such universal weak bialgebras are not likely to exist [Remark~\ref{rem:naive}]. This brings us to our main result.

\begin{theorem} \label{thm:main-intro}
For a finite quiver $Q$, the UQSGds $\Oleft(\kk Q)$, $\Oright(\kk Q)$, and $\Otrans(\kk Q)$ of the path algebra $\kk Q$ exist, and each is isomorphic to Hayashi's face algebra $\hay$ as weak bialgebras.
\end{theorem}

For example, if we take $A$ to be the (connected, graded) free algebra $\kk \langle t_1, \dots, t_n \rangle$, i.e., the path algebra on the $n$-loop quiver $Q_{n\text{-loop}}$, then the UQSGds of $A$ are the classical UQSGs of Definitions~\ref{def:UQSG-intro} and \ref{def:ManinUQSG-intro}, and 
$$O^{\textnormal{left}}(A) \;\cong\; O^{\textnormal{right}}(A)\; \cong \;O^{\textnormal{trans}}(A)\; \cong\; \mathfrak{H}(Q_{n\text{-loop}});$$
see Example~\ref{ex: Manincase}. But these isomorphisms need not hold if $A$ is a proper quotient of $\kk \langle t_1, \dots, t_n \rangle$ [Example~\ref{ex:polynomialManin}]. In general, we have the following results for UQSGds of graded quotient algebras of~$\kk Q$.

\begin{proposition} \label{prop:quotient-intro}
Let $I \subseteq \kk Q$ be a graded ideal which is generated in degree $2$ or greater. If $\Ocal^*(\kk Q/I)$ exists, then we have an isomorphism of weak bialgebras,
\[\Ocal^*(\kk Q/I)\; \cong \; \hay/\mathcal{I},
\]
for some biideal $\mathcal{I}$ of $\hay$. Here,  $*$ means `left', `right', or `trans'. 
\end{proposition}

Finally, in the case when $I$ is generated in degree 2, i.e., when $\kk Q/I$ is {\it quadratic} [Definition~\ref{def:quaddual}], we establish a non-connected generalization of \cite[Theorem~5.10]{Manin}. The {\it quadratic dual} $(\kk Q/I)^!$ of the quadratic algebra $\kk Q/I$ is reviewed in Definition~\ref{def:quaddual}.

\begin{theorem} \label{thm:quad-intro}
If the quotient algebra $\kk Q/I$ is quadratic, then we have that
\begin{enumerate} [(a),font=\upshape] 
    \item  $\Oleft(\kk Q/I) \cong  \Oright((\kk Q/I)^!)^{\op}$,
    \smallskip
    \item $\Oright(\kk Q/I) \cong \Oleft((\kk Q/I)^!)^{\op}$,
    \smallskip
    \item $\Oleft(\kk Q/I) \cong \Oright(\kk Q/I)^{\cop}$,
    \smallskip
    \item $\Otrans(\kk Q/I) \cong \Otrans((\kk Q/I)^!)^{\op}$,
\end{enumerate}
as weak bialgebras.
\end{theorem}

See Remark~\ref{rem:N-homog} for a discussion of  universal quantum semigroupoids of $N$-homogeneous algebras; such algebras include quiver potential algebras.

\smallskip

The paper is organized as follows. We present background material and preliminary results on weak bialgebras, monoidal categories of corepresentations of weak bialgebras, and (examples of) comodule algebras over weak bialgebras in Section~\ref{sec:prelim}. We introduce the theory of universal quantum linear semigroupoids (of algebras as in Hypothesis~\ref{hyp:A}) in Section~\ref{sec:universal}, including Definition~\ref{def:UQSGd-intro}. Our main result, Theorem~\ref{thm:main-intro} on the UQSGds of path algebras, is established in Section~\ref{sec:main}. Examples and results about UQSGds of quotients of path algebras are presented in Section~\ref{sec:quotients}, including Proposition~\ref{prop:quotient-intro} and Theorem~\ref{thm:quad-intro}. We end by providing directions for future investigation on {\it universal quantum linear groupoids} (i.e., universal weak Hopf algebras) in Section~\ref{sec:UQGd}.

\medskip

\noindent {\bf Acknowledgements.}
The authors would like to thank Dmitri Nikshych for a helpful exchange about material in Section~\ref{sec:prelim}, Pavel Etingof for posing Question~\ref{ques:A0comm} and other interesting comments,  James Zhang for inspiring Question~\ref{ques:ring-hom}, and Fabio Calder\'{o}n for helpful questions and comments. 
C. Walton is supported by a research grant from the Alfred P. Sloan foundation and by NSF grant \#DMS-1903192 and 2100756. R. Won is supported by an AMS--Simons Travel Grant.


\section{Preliminaries} \label{sec:prelim}

In this section, we provide background material and preliminary results on weak bialgebras [Section~\ref{sec:wba}], and on corepresentation categories of weak bialgebras and algebras within them [Section~\ref{sec:corep}]. We end by providing crucial examples of comodule algebras over weak bialgebras [Section~\ref{sec:ex}].

\subsection{Weak bialgebras} \label{sec:wba}
To begin, recall that a {\it $\kk$-algebra} is a $\kk$-vector space $A$ equipped with a multiplication map $m: A \otimes A \to A$ and unit map $u: \kk \to A$ satisfying associativity and unitality  constraints. We reserve the notation 1 to mean $1:= 1_A:= u(1_\kk).$
A {\it $\kk$-coalgebra} is a $\kk$-vector space $C$ equipped with a comultiplication map $\Delta: C \to C \otimes C$ and counit map $\ep: C \to \kk$ satisfying coassociativity and counitality  constraints. If $(C, \Delta, \varepsilon)$ is a coalgebra, we use sumless Sweedler notation and write $\Delta(c):=c_1 \otimes c_2$   for $c\in C$.

\begin{definition} \label{def:wba}
A \textit{weak bialgebra} over $\kk$ is a quintuple $(H,m,u,\Delta, \varepsilon)$ such that
\begin{enumerate}[label=(\roman*)]
    \item $(H,m,u)$ is a $\kk$-algebra,
    \item $(H, \Delta, \varepsilon)$ is a $\kk$-coalgebra,
    \item \label{def:wba3} $\Delta(ab)=\Delta(a)\Delta(b)$ for all $a,b \in H$,
    \item \label{def:wba4} $\varepsilon(abc)=\varepsilon(ab_1)\varepsilon(b_2c)=\varepsilon(ab_2)\varepsilon(b_1c)$ for all $a,b,c \in H$,
    \item \label{def:wba5} $\Delta^2(1)=(\Delta(1) \otimes 1)(1 \otimes \Delta(1))=(1 \otimes \Delta(1))(\Delta(1) \otimes 1)$.
\end{enumerate}
\end{definition}

The difference between a bialgebra and a weak bialgebra can be understood as a weakening of the compatibility between the algebra and coalgebra structures. In a weak bialgebra, we still have that comultiplication is multiplicative (e.g.,  condition \ref{def:wba3}), but the counit is no longer multiplicative and we do not necessarily have $\Delta(1)=1 \otimes 1$ or $\varepsilon(1)=1$. Instead, we have weak multiplicativity of the counit (condition \ref{def:wba4}) and weak comultiplicativity of the unit (condition \ref{def:wba5}). 

\begin{definition}[$\varepsilon_s$, $\varepsilon_t$, $H_s$, $H_t$] \label{def:eps}
Let $(H, m, u, \Delta, \varepsilon)$ be a weak bialgebra. We define the {\it source and target counital maps}, respectively as follows:
\[
\begin{array}{ll}
    \varepsilon_s: H \to H, & x \mapsto 1_1\ep(x1_2) \\
    \ept: H \to H,  & x \mapsto \ep(1_1x)1_2.
\end{array}
\]
We denote the images of these maps as $H_s:=\eps(H)$ and $H_t:=\ept(H)$. We call $H_s$ the \emph{source counital subalgebra} and $H_t$ the \emph{target counital subalgebra} of $H$ (see \cref{prop:Hs-facts}).
\end{definition}

These subalgebras have special properties that we will need below. 

\begin{proposition} \label{prop:Hs-facts} 
Let $H$ and K be  weak bialgebras. The following statements hold. 
\begin{enumerate}[(a), font=\upshape]
\item \label{itm:fd-sep} $H_s$ and $H_t$ are separable Frobenius (so, finite-dimensional) $\kk$-algebras. 

\item \label{itm:eset} $\eps(y) = y$ for $y \in H_s$, \;and\;
$\ept(z) =  z$ for $z \in H_t$.

\item \label{itm:HsHtcommute} If $y \in H_s$ and $z \in H_t$, then $yz=zy$.

\item \label{itm:deltaHsht} $\Delta(y) = 1_1 \otimes y 1_2 =  1_1 \otimes 1_2y$ for $y \in H_s$, \;and\;
$\Delta(z) = 1_1 z \otimes 1_2 =  z 1_1 \otimes 1_2$ for $z \in H_t$.

\item \label{itm:coideal} $H_s$ (resp., $H_t$) is a left (resp., right) coideal subalgebra of $H$. We also have that
\[H_t=\{(\varphi \otimes \id)\Delta(1): \varphi \in H^*\}, \quad H_s=\{(\id \otimes \varphi)\Delta(1): \varphi \in H^*\}.\]
 
\item \label{itm:anti-iso}$\ep_t$ is an anti-isomorphism from $H_s$ to $H_t$, i.e.  $H_s\cong H_t^{\textnormal{op}}$ as $\kk$-algebras.
    
\item\label{itm:bialgebra} $H$ is a bialgebra if and only if $\dim_\kk H_s =1$, if and only if  $\dim_\kk H_t =1$.

 \item \label{itm:preservesubalg}  Any nonzero weak bialgebra morphism $\alpha: H \rightarrow K$ preserves counital subalgebras, i.e. $H_s \cong K_s$ and $H_t
\cong K_t$ as $\kk$-algebras.
\end{enumerate}
\end{proposition}

\begin{proof}
\ref{itm:fd-sep} This follows from \cite[Corollary~4.4]{BCJ} and \cite[Proposition~2.11]{BNS}.

\smallskip

\ref{itm:eset}, \ref{itm:HsHtcommute}, \ref{itm:deltaHsht},  \ref{itm:coideal} These parts follow from 
\cite[Section~2.2]{BNS} and \cite[Propositions~2.2.1 and~2.2.2]{NV}.

\smallskip

\ref{itm:anti-iso} This is an immediate consequence of \cite[Propositions~1.15 and 1.18]{BCJ}.

\smallskip

\ref{itm:bialgebra} This is standard, and follows from \ref{itm:anti-iso} and  \cite[Definition~3.1, Remark~3.2]{Nik02}, for instance.

\smallskip

\ref{itm:preservesubalg}
The result for weak Hopf algebras is provided in \cite[Proposition~2.3.3]{NV}, and we generalize this to weak bialgebras as follows. Write $\Delta(1_H)=\sum_{i=1}^nw_i\otimes z_i$ with $\{w_i\}_{i=1}^n$ and $\{z_i\}_{i=1}^n$ linearly independent. By part~\ref{itm:coideal}, $H_s = \text{span}_\kk \{w_i\}_{i=1}^n$. Using the linear independence of $\{w_i\}_{i=1}^n$, we have
\begin{equation}\label{eq:n}
    n = \dim_\kk H_s.
\end{equation}  
Since
$z_{j} \overset{\textnormal{\ref{itm:eset}}}{=} \ep_{t}\left(z_{j}\right)\overset{\textnormal{\ref{def:eps}}}{=}\sum_{i=1}^{n} \ep\left(w_{i} z_{j}\right) z_{i}$ and $\{z_i\}_{i=1}^n$ are linearly independent, we also have
 \begin{equation}
 \label{ep:w_iz_j}
  \ep\left(w_{i} z_{j}\right)= \delta_{i,j}.   
 \end{equation}
Therefore
$$
\operatorname{dim}_\kk H_{s} \overset{\eqref{eq:n}}{=} n 
\overset{\eqref{ep:w_iz_j}}{=}  \textstyle \sum_{j=1}^n \ep_H(w_j z_j)
= \ep_H( (1_H)_1 \; (1_H)_2)
\overset{(*)}{=} \ep_K( (1_K)_1 \; (1_K)_2)
= \operatorname{dim}_\kk K_{s}.$$
Here, $(*)$ holds because the nonzero map $\alpha: H\to K$ is an algebra and a coalgebra map; that is, $1_K = u_K(1_\kk) = \alpha u_H (1_\kk) = \alpha(1_H)$ and 
$$\ep_K\; m_K \;\Delta_K(1_K)  ~=~  \ep_K \; m_K\; (\alpha \otimes \alpha) \; \Delta_H(1_H) ~=~  \ep_K \; \alpha\; m_H \; \Delta_H(1_H)  ~=~  \ep_H \; m_H \;\Delta_H(1_H).$$
Moreover, since $\alpha$ is a coalgebra map, \[
\textstyle \Delta(1_{K})=\sum_{i=1}^{n} \alpha(w_{i}) \otimes \alpha(z_{i}).
\]
By part~\ref{itm:coideal}, $K_{s}=\operatorname{span}\left\{\alpha(w_i)\right\},$ i.e., $\left.\alpha \right|_{H_{s}}: H_{s}\longrightarrow K_{s}$ is  a surjective algebra morphism. Thus, $\left. \alpha \right|_{H_{s}}$ is bijective. The proof for target subalgebras is similar. See also \cite[Lemma~6.3]{Schauenburg-wha} for an alternative proof.
 \end{proof}

In this paper, the main weak bialgebras of interest are the following examples due to Hayashi, see, e.g., \cite[Example~1.1]{Hayashi96}. Recall Notation~\ref{not:quiver}.

\begin{example}[Hayashi's face algebra attached to a quiver] 
\label{ex:hay}
For a finite quiver $Q$, we define the weak bialgebra $\mfH(Q)$ as follows. As a $\kk$-algebra, $$\mfH(Q) = \frac{\kk \l<x_{i,j}, x_{p,q} \mid i,j \in Q_0, \;  p,q \in Q_1\r>}{(R)},$$
for indeterminates $x_{i,j}$ and $x_{p,q}$ with relations $R$, given by:
\begin{equation}
\label{eq:multHQ1}
 x_{p,q} x_{p',q'} 
= \delta_{t(p), s(p')} \delta_{t(q), s(q')} x_{p,q} x_{p',q'} ,
\end{equation}
\begin{equation}
\label{eq:multHQ01}
 x_{s(p),s(q)}x_{p,q}=x_{p,q}=x_{p,q} x_{t(p),t(q)}, 
\end{equation}
for all $p,p',q,q' \in Q_1$, and
\begin{equation}
\label{eq:multHQ0}
 x_{i,j} x_{k,{\ell}} 
= \delta_{i,k} \delta_{j,{\ell}}  x_{i,j} 
\end{equation}
for all $i,j,k,{\ell} \in Q_0$. (In fact, \eqref{eq:multHQ1} follows from \eqref{eq:multHQ01} and \eqref{eq:multHQ0}.) Then $\hay$ is a unital $\kk$-algebra, with unit given by
\begin{align}
\label{eq:unitHQ}
1_{\hay} = \textstyle \sum_{i, j \in Q_0} x_{i,j}.
\end{align}
Let $k \geq 2$ and suppose that $p_1p_2\cdots p_k, \: q_1q_2\cdots q_k \in Q_k$, where each $p_i,q_i \in Q_1$. As shorthand, we define the symbols
\begin{equation}
\label{eq:symbHQ}
    x_{p_1\cdots p_k, q_1\cdots q_k}:=x_{p_1,q_1}x_{p_2,q_2}\cdots x_{p_k,q_k}.
\end{equation}
With this notation, as a vector space we can write
\[\hay= \textstyle \bigoplus_{\ell \geq 0} \bigoplus_{a,b \in Q_{\ell}} \kk x_{a,b}.\]
For $a,b \in Q_{\ell}$, the coalgebra structure is given by
\begin{equation}
\label{eq:coalgHQ}
    \Delta(x_{a,b}) =  \textstyle \sum_{c \in Q_{\ell}} x_{a,c} \otimes x_{c,b}
\quad \quad \text{and} \quad \quad \ep(x_{a,b}) = \delta_{a,b}.
\end{equation} 
It can be checked that this structure makes $\hay$  a weak bialgebra.
\end{example}

We record the following facts about $\hay$.

\begin{proposition}
\label{prop:epHQ}
Let $Q$ be a finite quiver.
\begin{enumerate}[(a),font=\upshape]
    \item \label{eq:epHQ} For $p_1, \dots, p_k, q_1, \dots, q_k \in Q_1$,
\begin{equation*}
\ep(x_{p_1,q_1}\cdots x_{p_k,q_k})\\
=\l(\delta_{t(p_1),s(p_2)} \cdots \delta_{t(p_{k-1}),s(p_k)}\r)\l(\delta_{t(q_1),s(q_2)} \cdots \delta_{t(q_{k-1}),s(q_k)}\r)\l(\delta_{p_1,q_1}\cdots \delta_{p_k,q_k}\r).
\end{equation*}
\item For each $j \in Q_0$, define
\[ a_j = \textstyle \sum_{i\in Q_0} x_{i,j} \quad \text{and} \quad a_j' = \textstyle \sum_{i \in Q_0} x_{j,i}.
\]
Then $\{a_j\}_{j \in Q_0}$ and $\{a_j'\}_{j\in Q_0}$ are complete sets of primitive orthogonal idempotents in $\hay$ called the `face idempotents' (see \cite{Hayashi93}).
\smallskip
\item \label{prop:epHQs} As $\kk$-vector spaces, $\mfH(Q)_s = \bigoplus_{j \in Q_0} \kk a_j$ and $\mfH(Q)_t = \bigoplus_{j \in Q_0} \kk a'_j$. 
\end{enumerate}
\end{proposition}

\begin{proof}
(a) The equation clearly holds for $k=1$. We will show this for $k=2$; the rest follows by induction:
\begin{align*}
    \ep(x_{p_1,q_1}x_{p_2,q_2}) &\overset{\eqref{eq:multHQ1}}{=}\ep(\delta_{t(p_1),s(p_2)}\delta_{t(q_1),s(q_2)}x_{p_1,q_1}x_{p_2,q_2}) \\
    &\overset{\eqref{eq:symbHQ}}{=}\delta_{t(p_1),s(p_2)}\delta_{t(q_1),s(q_2)}\ep(x_{p_1p_2,q_1q_2}) \\
    &\overset{\eqref{eq:coalgHQ}}{=}\delta_{t(p_1),s(p_2)}\delta_{t(q_1),s(q_2)}\delta_{p_1p_2,q_1q_2} \\
    &\overset{}{=}\delta_{t(p_1),s(p_2)}\delta_{t(q_1),s(q_2)}\delta_{p_1,q_1}\delta_{p_2,q_2}.
\end{align*}

\smallskip

(b) This is straightforward to check.
\smallskip

(c) 
We get $\eps(x_{a,b}) = \delta_{a,b} \sum_{i \in Q_0} x_{i,t(a)}$ and $\ept(x_{a,b}) = \delta_{a,b} \sum_{j \in Q_0} x_{s(b),j}$ for $a,b\in Q_{\ell}$. 
\end{proof}


\subsection{Corepresentation categories of weak bialgebras} \label{sec:corep}

Here, we discuss the monoidal categories of corepresentations of weak bialgebras, and algebras within these categories.

\begin{definition}
A \textit{monoidal category} $\mathcal{C}=(\mathcal{C}, \otimes, \unit, \alpha, l, r)$ consists of:   a category $\CC$; a bifunctor $\otimes : \CC\times \CC\rightarrow \CC$; a natural isomorphism   $\alpha_{X,Y,Z}: (X \otimes Y)\otimes Z \overset{\sim}{\to} X \otimes (Y\otimes Z)$ for each $X,Y,Z \in \CC$;  an object $\unit \in \CC$; and natural isomorphisms $l_X: \mathbbm{1} \otimes X \overset{\sim}{\to} X, \:\: r_X:X \otimes \mathbbm{1} \overset{\sim}{\to} X$ for each $X \in \CC$,
 such that the pentagon and triangle axioms are satisfied (see \cite[Equations~2.2, 2.10]{EGNO}).
\end{definition}

An example of a monoidal category is $\Vec$, the category of finite-dimensional $\kk$-vector spaces, with $\otimes = \otimes_\kk$, $\unit = \kk$, and with the canonical associativity and unit isomorphisms.
If $H$ is a weak bialgebra, we can endow the category of right (or left) $H$-comodules with the structure of a monoidal category as follows.

\begin{example}[{\cite{BCJ,Nill}}]
\label{ex:right}
For a weak bialgebra $H = (H, m, u, \Delta, \varepsilon)$, the category $\MH$ of right $H$-comodules can be given the structure of a monoidal category:
\[\MH = ({\sf Comod}\text{-}H,  \hspace{.05in} \tenbar, \hspace{.05in} \unit = H_s, \hspace{.05in} \alpha = \alpha_{{\sf Vec}_\Bbbk}, \hspace{.05in} l, \hspace{.05in} r).
\]
Here, for $M,N \in \MH$, the monoidal product of $M$ and $N$ is defined to be
\begin{equation*} \label{eq:tenbar}
 M \tenbar N := \left\{m \otimes n \in M \otimes N \mid m \otimes n = \varepsilon(m_{[1]}n_{[1]}) m_{[0]} \otimes n_{[0]} \right\}.
\end{equation*}
The counital subalgebra $H_s$ is naturally a right $H$-comodule since the image of $\Delta|_{H_s}$ is a subspace of $H_s \otimes H$, and so $\Delta|_{H_s}$ can be viewed as a map $H_s \to H_s \otimes H$. By \cite[Theorem 3.1]{BCJ}, $H_s$ is the unit object of the monoidal category $\MH$. By \cite[Section~3]{BCJ}, the monoidal category $\MH$ has unit isomorphisms:
$$l_M: H_s \tenbar M \to M, \quad \quad x \; \bar{\otimes} \; m = \varepsilon(x_{[1]}m_{[1]}) x_{[0]} \tenbar m_{[0]} ~~\mapsto~~ \varepsilon(x m_{[1]}) m_{[0]},$$

\vspace{-.2in}

$$r_M:  M \tenbar H_s \to M, \quad \quad m \; \bar{\otimes} \; x = \varepsilon(m_{[1]}x_{[1]}) m_{[0]} \tenbar x_{[0]} ~~\mapsto~~ \varepsilon(m_{[1]} x) m_{[0]},$$
for all $M \in \MH$.  
\end{example}

\begin{example}
\label{ex:left}
Likewise, for a weak bialgebra $H = (H, m, u, \Delta, \varepsilon)$, the category $\HM$ of left $H$-comodules can be given the structure of a monoidal category:
\[\HM = (H\text{-}{\sf Comod},  \hspace{.05in} \tenbar, \hspace{.05in} \unit = H_t, \hspace{.05in} \alpha = \alpha_{{\sf Vec}_\Bbbk}, \hspace{.05in} l, \hspace{.05in} r).
\]
To the best of our knowledge, the details of the monoidal structure of this category are not explicitly stated in the literature, so we include them for the convenience of the reader. For $M,N \in \HM$, the monoidal product of $M$ and $N$ is defined to be 
\begin{equation*} \label{eq:tenbarleft}
 M \tenbar N := \left\{m \otimes n \in M \otimes N \mid m \otimes n = \varepsilon(m_{[-1]}n_{[-1]}) m_{[0]} \otimes n_{[0]} \right\}.
\end{equation*}
The restriction of the coproduct $\Delta|_{H_t}$, viewed as a map $H_t \to H \otimes H_t$ makes $H_t$ a left $H$-comodule which is the unit object of the monoidal category $\HM$. Explicitly, the unit isomorphisms of $\HM$ are given by:
\[l_M: H_t \tenbar M \to M, \quad \quad x \; \bar{\otimes} \; m = \varepsilon(x_{[-1]}m_{[-1]}) x_{[0]} \tenbar m_{[0]} ~~\mapsto~~ \varepsilon(x m_{[-1]}) m_{[0]}
\]
\vspace{-.2in}
\[r_M:  M \tenbar H_t \to M, \quad \quad m \; \bar{\otimes} \; x = \varepsilon(m_{[-1]}x_{[-1]}) m_{[0]} \tenbar x_{[0]} ~~\mapsto~~ \varepsilon(m_{[-1]} x) m_{[0]},
\]
for all $M \in \HM$. 
\end{example}

Now we turn our attention to algebras  in monoidal categories.

\begin{definition}[${\sf Alg}(\CC)$] 
\label{def:alg} Let $(\CC, \otimes, \unit, \alpha, l, r)$ be a  monoidal category.
An \textit{algebra} in $\CC$ is a triple $(A,m,u)$, where $A$ is an object in $\CC$, and  $m:A \otimes A \to A$, $u:\unit \to A$ are morphisms in $\CC$, satisfying unitality and associativity constraints: 
    $$ m (m \otimes \id) = m(\id \otimes m) \alpha_{A,A,A},\quad 
       m (u \otimes \id) = l_A,  \quad m(\id \otimes u) = r_A. $$
A {\it morphism} of algebras $(A, m_A, u_A)$ to $(B, m_{B}, u_{B})$  is a morphism $f: A \to B$ in $\CC$ so that $fm_A = m_{B \otimes B}(f \otimes f)$ and $fu_A = u_{B}$.  Algebras in $\CC$ and their morphisms  form a category, which we denote by ${\sf Alg}(\CC)$.
\end{definition}

Algebras in $\Vec$ are the same as $\kk$-algebras. 

\smallskip

Now we consider algebras that have the structure of a comodule over a weak bialgebra $H$. 
There are two related notions: we can consider the objects in ${\sf Alg}(\MH)$ (or, ${\sf Alg}(\HM)$), or we can consider $\kk$-algebras (i.e., objects of ${\sf Alg}(\Vec)$) which are also right (or, left) $H$-comodules such that the algebra and comodule structures are compatible as done below. 

\begin{definition}[$\HA$, $\mathcal{A}^H$]
\label{def:Halg}
Let $H$ be a weak bialgebra. 
\begin{enumerate}[(a)]
    \item \label{def:Halg-L} Consider the {\it category $\HA$ of left $H$-comodule algebras} defined as follows. The objects of $\HA$ are objects of ${\sf Alg}(\Vec)$, 
$$(A,~m_A: A \otimes A \to A, ~u_A: \Bbbk \to A),$$
with $1_A:=u_A(1_\Bbbk)$, so that the $\Bbbk$-vector space $A$ is a left $H$-comodule via $$\lambda_A: A \to H \otimes A, \quad a \mapsto a_{[-1]} \otimes a_{[0]},$$ the multiplication map $m_A$ is compatible with $\lambda_A$ in the sense that 
\begin{equation}
\label{def:Halg-L-mult}
  (ab)_{[-1]} \otimes (ab)_{[0]} = a_{[-1]} b_{[-1]} \otimes a_{[0]} b_{[0]} \quad  \forall a,b \in A; 
\end{equation} 
the unit map $u_A$ is compatible with $\lambda_A$ in the sense that  
\begin{equation} 
\label{def:Halg-L-unit}
\lambda_A(1_A) \in H_s \otimes A.  
\end{equation}
The morphisms of $\HA$ are maps in ${\sf Alg}(\Vec)$ that are also $H$-comodule maps. 

\medskip

    \item \label{def:Halg-R} Consider the {\it category $\mathcal{A}^H$ of right $H$-comodule algebras} defined as follows. The objects of $\mathcal{A}^H$ are objects of ${\sf Alg}(\Vec)$, 
$$(A,~m_A: A \otimes A \to A, ~u_A: \Bbbk \to A),$$
with $1_A:=u_A(1_\Bbbk)$, so that the $\Bbbk$-vector space $A$ is a right $H$-comodule via $$\rho_A: A \to A \otimes H, \quad a \mapsto a_{[0]} \otimes a_{[1]},$$ the multiplication map $m_A$ is compatible with $\rho_A$ in the sense that 
$$(ab)_{[0]} \otimes (ab)_{[1]} = a_{[0]} b_{[0]} \otimes a_{[1]} b_{[1]} \quad  \forall a,b \in A; $$
the unit map $u_A$ is compatible with $\rho_A$ in the sense that  
$$ \rho_A(1_A) \in A \otimes H_t.$$
The morphisms of $\mathcal{A}^H$ are maps in ${\sf Alg}(\Vec)$ that are also $H$-comodule maps. 
\end{enumerate}
\end{definition}

The categories $\HA$ and $\Alg(\HM)$ (likewise, $\AH$ and $\Alg(\MH))$ are essentially the same.

\begin{proposition}\cite[Theorem~4.4]{WWW} \label{prop:alg}
There is an isomorphism of categories between ${\sf Alg}(\MH)$ and $\mathcal{A}^H$, and between ${\sf Alg}(\HM)$ and $\HA$. \qed
\end{proposition}

In \cite[Theorem 4.5]{WWW}, the functors between ${\sf Alg}(\MH)$ and $\mathcal{A}^H$ are given explicitly. For the isomorphism between ${\sf Alg}(\HM)$ and $\HA$, the proof should be adjusted using the structures in Example~\ref{ex:left} rather than Example~\ref{ex:right}.

\smallskip

\subsection{Examples} \label{sec:ex}
Now we provide some examples of comodule algebras over weak bialgebras, which will be important in the rest of the paper.

\begin{example} \label{ex:haycoaction}
Consider Hayashi's face algebra $\hay$ from Example~\ref{ex:hay}. It is straightforward to check that the path algebra $\kk Q$ belongs to ${}^{\hay} \mathcal{A}$ and to $\mathcal{A}^{\hay}$ via the coactions:
\begin{align*}
\lambda: \kk Q &\to \hay \otimes \kk Q & \rho: \kk Q &\to \kk Q \otimes \hay \\
 e_j &\mapsto \textstyle \sum_{i \in Q_0} x_{j,i} \otimes e_i & e_j &\mapsto \textstyle \sum_{i \in Q_0}  e_i \otimes x_{i,j}\\
 q &\mapsto \textstyle \sum_{p \in Q_1} x_{q,p} \otimes p  & q &\mapsto \textstyle \sum_{p \in Q_1}  p\otimes x_{p,q},
\end{align*}
for $j \in Q_0$ and $q \in Q_1$. See \cite[Example~4.10]{WWW} for verification that $\kk Q \in \mathcal{A}^{\hay}$.
\end{example}

\begin{example}
\label{ex:D}
Let $Q_{\text{\tiny{$\bullet \bullet$}}}$ be the quiver with two vertices and no arrows. Let $D$ be the algebra
\[ D = \frac{\kk \langle x, y \rangle }{\left(x^2 = x, \: y^2 = y, \: xy = yx = 0 \right)}
\]
so that $1_D = x + y$ (as an algebra, $D \cong \kk Q_{\text{\tiny{$\bullet \bullet$}}}$). Define a coproduct $\Delta_D$ on $D$ by
\[ \Delta_D(x) = x \otimes x + y \otimes y, \quad \quad \Delta_D(y) = x \otimes y + y \otimes x
\]
and a counit $\varepsilon_D$ by
\[ \varepsilon_D(x) = 1_\kk, \quad \quad \varepsilon_D(y) = 0_\kk.
\]
One can verify that this makes $D$ a bialgebra.

One can  show that $\kk Q_{\text{\tiny{$\bullet \bullet$}}}$ is  a transposed $D$-comodule algebra [Definition~\ref{def:ManinUQSG-intro}\ref{itm:transpose-intro}] under the left and right coactions: 
\[
\begin{array}{lll}
\medskip
\lambda: \kk Q_{\text{\tiny{$\bullet \bullet$}}} \to D \otimes \kk Q_{\text{\tiny{$\bullet \bullet$}}}, \quad
&e_1 \mapsto x \otimes e_1 + y \otimes e_2, \quad
&e_2 \mapsto y \otimes e_1 + x \otimes e_2;\\
\rho: \kk Q_{\text{\tiny{$\bullet \bullet$}}} \to \kk Q_{\text{\tiny{$\bullet \bullet$}}} \otimes D, \quad  &e_1 \mapsto e_1 \otimes x + e_2 \otimes y,  \quad &e_2 \mapsto e_1 \otimes y + e_2 \otimes x. 
\end{array}
\]
\end{example}
\smallskip

For our next example, we will need the following two lemmas. These lemmas are well-known, and their proofs are routine.

\begin{lemma}
\label{lem:H+K}
If $(H,m_H, u_H, \Delta_H, \ep_H)$ and $(K,m_K,u_K,\Delta_K,\ep_K)$ are weak bialgebras, then $H \oplus K$ is a weak bialgebra with the following structure for all $h,g \in H, k,l \in K$:
\begin{align*}
 &\text{multiplication:} &(h,k)(g,l):=(hg,kl); \\
 &\text{unit:} & 1_{H\oplus K}:=(1_H,1_K);  \\
 &\text{comultiplication:} &\Delta_{H\oplus K}((h,k)):=(h_1,0)\otimes (h_2,0) + (0,k_1) \otimes (0,k_2); \\
 &\text{counit:} & \ep_{H\oplus K}((h,k)):=\ep_H(h)+\ep_K(k). 
\end{align*}
We also have that
\[(\ep_{H\oplus K})_t(h,k) = \l((\ep_H)_t(h), \;(\ep_K)_t(k)\r), \quad (\ep_{H\oplus K})_s(h,k) = \l((\ep_H)_s(h), \;(\ep_K)_s(k)\r). \qed\]

\end{lemma}

\medskip

\begin{lemma}
\label{lem:HK-comod}
Suppose that $V$ is a right $H$-comodule via
\[\rho_H: V \to V \otimes H, \quad v \mapsto v_{[0]} \otimes v_{[1]}.\]
Then $V$ is a right $(H \oplus K)$-comodule via
\[\rho: V \to V \otimes (H\oplus K), \quad v \mapsto v_{[0]} \otimes (v_{[1]},0).\]
Furthermore, if $V$ is a right $H$-comodule algebra via $\rho_H$, then $V$ is a right $(H\oplus K)$-comodule algebra via $\rho$.
A similar statement holds for left $H$-comodules and left $H$-comodule algebras. \qed
\end{lemma}


\begin{example}\label{ex1} 
Let $Q_{\text{\tiny{$\bullet \bullet$}}}$ be the quiver with two vertices and no arrows, and recall the bialgebra $D$ defined in \cref{ex:D}. A presentation of $D$ is given by
\[D:=\frac{\kk\l<y_{1,1},\;y_{1,2},\;y_{2,1},\;y_{2,2}\r>}{\l(y_{1,1}=y_{2,2},\; y_{1,2}=y_{2,1},\; y_{i,j}y_{i,k}=\delta_{j,k}y_{i,j}, \;y_{j,i}y_{k,i}=\delta_{j,k}y_{j,i}\r)},\]
with unit $1_D = y_{1,1} + y_{1,2}$. 

\medskip

\noindent \underline{Claim 1}. $\kk Q_{\text{\tiny{$\bullet \bullet$}}}$ is a left and right $(D \oplus D)$-comodule algebra via linear coactions.

\smallskip

\noindent {\it Proof of Claim 1}. The coalgebra structure is given by
\[\textstyle \Delta_D(y_{i,j}) = \sum_{k \in (Q_{\text{\tiny{$\bullet \bullet$}}})_0} y_{i,k} \otimes y_{k,j}, \quad \quad \ep_D(y_{i,j})=\delta_{i,j}, \quad \text{for all }i,j \in (Q_{\text{\tiny{$\bullet \bullet$}}})_0.\]
With this presentation, $D$ left and right coacts linearly on $\kk Q_{\text{\tiny{$\bullet \bullet$}}}$ via
\[
\begin{array}{ll}
\medskip
\textstyle \kk Q_{\text{\tiny{$\bullet \bullet$}}} \to D \otimes \kk Q_{\text{\tiny{$\bullet \bullet$}}}, \quad  &e_i \mapsto \sum_{j\in (Q_{\text{\tiny{$\bullet \bullet$}}})_0} y_{i,j} \otimes e_j\\
\textstyle \kk Q_{\text{\tiny{$\bullet \bullet$}}} \to \kk Q_{\text{\tiny{$\bullet \bullet$}}} \otimes D, \quad &e_i \mapsto \sum_{j\in (Q_{\text{\tiny{$\bullet \bullet$}}})_0} e_j \otimes y_{j,i}.
\end{array}
\]
By \cref{lem:HK-comod} and Example~\ref{ex:D}, we have that the coactions
\begin{align*}
   &\lambda: \kk Q_{\text{\tiny{$\bullet \bullet$}}} \to (D\oplus D) \otimes \kk Q_{\text{\tiny{$\bullet \bullet$}}} & \rho: \kk Q_{\text{\tiny{$\bullet \bullet$}}} \to \kk Q_{\text{\tiny{$\bullet \bullet$}}} \otimes (D \oplus D) \\
   &\textstyle  e_i \mapsto \sum_{j\in (Q_{\text{\tiny{$\bullet \bullet$}}})_0} (y_{i,j},0) \otimes e_j &\textstyle  e_i \mapsto \sum_{j\in (Q_{\text{\tiny{$\bullet \bullet$}}})_0} e_j \otimes (y_{j,i},0)
\end{align*}
yield the claim. \qed

\medskip

\noindent \underline{Claim 2}. $(D \oplus D)_t \cong \kk (Q_{\text{\tiny{$\bullet \bullet$}}})_0$ as algebras over $\kk$.

\smallskip

\noindent {\it Proof of Claim 2}.
Consider the morphism
$$\psi: \kk (Q_{\text{\tiny{$\bullet \bullet$}}})_0 \to (D \oplus D)_t,  \quad e_1 \mapsto (1_D,0), ~ e_2 \mapsto (0,1_D).$$
First, we will show that as a $\kk$-vector space, $(D \oplus D)_t = \text{Span}_{\kk}\l\{(1_D,0),\;(0,1_D)\r\}.$ By \cref{lem:H+K}, we have
\[(\ep_{D\oplus D})_t(1_D,0)=((\ep_D)_t(1_D),0)=(1_D,0), \qquad (\ep_{D\oplus D})_t(0,1_D)=(0,(\ep_D)_t(1_D))=(0,1_D).\]
Therefore, $\text{Span}_{\kk}\l\{(1_D,0),\;(0,1_D)\r\} \subseteq (D \oplus D)_t.$ To show the reverse inclusion, note that for $a,b \in D$ we have
\[(\ep_{D\oplus D})_t(a,b)=((\ep_D)_t(a),\;(\ep_D)_t(b))=(\ep_D(a)1_D, \;\ep_D(b)1_D),\]
where the last equality holds because $D$ is a bialgebra. 
Thus, $\kk (Q_{\text{\tiny{$\bullet \bullet$}}})_0 \cong \dim (D\oplus D)_t$ as $\kk$-vector spaces; here, $\dim (D\oplus D)_t = \dim \kk (Q_{\text{\tiny{$\bullet \bullet$}}})_0 =  2$. It is also clear that $\psi$ preserves the unit and multiplication. Therefore, $\psi$ is an isomorphism of $\kk$-algebras.

\medskip

\noindent \underline{Claim 3}. $\kk (Q_{\text{\tiny{$\bullet \bullet$}}})_0\not\cong (D\oplus D)_t  $ as left $(D\oplus D)$-comodules, where $\kk (Q_{\text{\tiny{$\bullet \bullet$}}})_0$ is a left $(D\oplus D)$-comodule via Claim~1, and $(D\oplus D)_t$ is naturally a left $(D\oplus D)$-comodule via comultiplication [\cref{ex:left}].

\smallskip

\noindent {\it Proof of Claim 3}.
By way of contradiction, suppose that we have an isomorphism \linebreak $\varphi: \kk (Q_{\text{\tiny{$\bullet \bullet$}}})_0 \to (D \oplus D)_t$ of left $(D\oplus D)$-comodules. Explicitly, the comodule structures are given by
\[\textstyle \lambda: \kk (Q_{\text{\tiny{$\bullet \bullet$}}})_0 \to (D \oplus D) \otimes \kk (Q_{\text{\tiny{$\bullet \bullet$}}})_0,  \quad  \quad
    e_i \mapsto \sum_{j\in (Q_{\text{\tiny{$\bullet \bullet$}}})_0}  (y_{i,j},0) \otimes e_j,\]
    
\vspace{-.4in}

\[
\begin{array}{crl}
\medskip
    &\\ &\lambda_t:=\Delta_{D\oplus D}|_{(D\oplus D)_t}:(D\oplus D)_t \to (D\oplus D) \otimes (D\oplus D)_t,   &(1_D,0) \mapsto (1_D,0) \otimes (1_D,0),\\ 
   & {}  &(0,1_D) \mapsto (0,1_D) \otimes (0,1_D).
\end{array}
\]
Since $(D\oplus D)_t = \text{Span}_{\kk}\{(1_D,0),\;(0,1_D)\},$ (see proof of Claim~2), we can write
\[\varphi(e_i)=\alpha_i(1_D,0) + \beta_i(0,1_D),\]
for some $\alpha_i,\beta_i \in \kk$. Since $\varphi$ is a left $(D \oplus D)$-comodule map,  $(\id_{(D\oplus D)} \otimes \varphi )\lambda=\lambda_t\varphi$. In particular,
\begin{align*}
   \textstyle \sum_{j \in (Q_{\text{\tiny{$\bullet \bullet$}}})_0} (y_{i,j},0) \otimes \l(\alpha_j(1_D,0)+\beta_j(0,1_D)\r) &= \textstyle \sum_{j\in (Q_{\text{\tiny{$\bullet \bullet$}}})_0}  (y_{i,j},0) \otimes \varphi(e_j)  \\
    &=\textstyle (\id_{D\oplus D} \otimes \varphi)\l(\sum_{j\in (Q_{\text{\tiny{$\bullet \bullet$}}})_0}  (y_{i,j},0) \otimes e_j\r) \\
    &= (\id_{D\oplus D} \otimes \varphi)\lambda(e_i) \\
    &=\lambda_t \varphi(e_i) \\
    &=\lambda_t\l(\alpha_i(1_D,0) + \beta_i(0,1_D)\r) \\
    &=\alpha_i(1_D,0)\otimes(1_D,0) + \beta_i(0,1_D)\otimes(0,1_D).
\end{align*}
Notice that the left hand side is contained in $(D \oplus 0) \otimes (D \oplus D)_t$. Therefore, we must have that $\beta_i=0$, since if not, the right hand side is not contained in $(D \oplus 0) \otimes (D \oplus D)_t$.  Therefore, $\varphi$ is not surjective and not an isomorphism of $(D\oplus D)$-comodules. \qed
\end{example}


\section{Universal linear coactions on graded algebras} 
\label{sec:universal}

In this section, we introduce the notion of a weak bialgebra that coacts linearly and universally on a graded algebra $A$ as in Hypothesis~\ref{hyp:A}. The universal weak bialgebras coacting on $A$ are defined below in Definitions~\ref{def:leftright-uqsgd} and~\ref{def:manin-uqsgd} below; we call them {\it universal quantum linear semigroupoids}. Recall here that $A$ is $\mathbb{N}$-graded $\kk$-algebra with $\dim_\kk A_i < \infty$ for all $i \in \mathbb{N}$, such that $A_0$ is a commutative, separable (so, Frobenius) $\kk$-algebra (we discuss how the assumptions on $A_0$ are used in Remarks~\ref{rem:A0fd-sep} and~\ref{rem:A0comm} below). Moreover, we say that $A$ is {\it connected} if $A_0 = \kk$, and that $A$ is {\it non-connected} otherwise. 

\smallskip

To proceed, we reinterpret the standing assumption, Hypothesis~\ref{hyp:linear-intro} from the introduction, as follows.

\begin{hypothesis} \label{hyp:co/action} 
[$\lambda, \lambda_i, \rho, \rho_i$] Let $H$ be a weak bialgebra, and recall the notion of a $H$-comodule algebra from  Definition~\ref{def:Halg}. From now on, we impose the assumptions below.

\begin{enumerate}
    \item Each left $H$-comodule algebra structure on $A$ will be {\it linear} in the sense that, for the structure map $\lambda:=\lambda^H_A: A \to H \otimes A$, the restriction
$\lambda|_{A_i}:=\lambda_i$
makes $A_i$ a left $H$-comodule for each $i$.

\smallskip

    \item Each right $H$-comodule algebra structure on $A$ will be {\it linear} in the sense that, for the structure map $\rho:=\rho^H_A: A \to A \otimes H$, the restriction
$\rho|_{A_i}:=\rho_i$
makes $A_i$ a right $H$-comodule for each $i$.
\end{enumerate}
\end{hypothesis}

\begin{remark} \label{rem:A0}
If $H$ left coacts linearly via $\lambda$ on $A$, then $A_0$ is a left $H$-comodule algebra  via $\lambda_0$. By \cite[Theorem~4.5]{WWW}, we can view $A_0$ as an object in the category  $\HA$ [Definition~\ref{def:Halg}]. The analogous statement holds for right coactions.
\end{remark}

Next, we discuss a na\"{i}ve notion of a weak bialgebra  coacting universally on $A$, that is, by merely replacing `bialgebra' with `weak bialgebra' in the definition of a universal quantum linear semigroup [Definition~\ref{def:UQSG-intro}]. This weak bialgebra fails to exist, even for an easy example of non-connected graded algebra $A$, as seen below.

\begin{remark} \label{rem:naive} 
Let $A$ be an algebra satisfying Hypothesis~\ref{hyp:A}.
Suppose that there exists a weak bialgebra $U := U(A)$ that left coacts on $A$ so that, for every weak bialgebra $H$ that left coacts on $A$, there exists a unique weak bialgebra map $\pi: U \to H$ so that $(\pi \otimes \id_A)\lambda^U = \lambda^H$. We will show that in general, such a weak bialgebra fails to exist.

Let $H$ be any nonzero weak bialgebra which left coacts on $A$. Then since there exists a weak bialgebra map $\pi: U \to H$, we have that $\dim_\kk U_t = \dim_\kk H_t$ by Proposition~\ref{prop:Hs-facts}\ref{itm:preservesubalg}.
Now take $A =  \kk Q_{\text{\tiny{$\bullet \bullet$}}}$ as in Section~\ref{sec:ex}, which is 
a comodule algebra over both the bialgebra $D$ (Example~\ref{ex:D}) and also over the weak bialgebra $D\oplus D$ (Example~\ref{ex1}).
By the above, if we take $H=D$, then we have $\dim_\kk U_t = \dim_\kk D_t$, and so by Proposition~\ref{prop:Hs-facts}\ref{itm:bialgebra}, $\dim_\kk U_t = 1$. 
On the other hand, we can also substitute $H$ by $D\oplus D$ and have $\dim_\kk U_t = \dim_\kk (D \oplus D)_t$; by Claim~2 of Example~\ref{ex1}, $\dim_\kk (D \oplus D)_t = \dim_\kk \kk (Q_{\text{\tiny{$\bullet \bullet$}}})_0 = 2$. Hence $1 = \dim_\kk U_t = 2$, which is a contradiction. Hence, $U(\kk Q_{\text{\tiny{$\bullet \bullet$}}})$  does not exist.
\end{remark}

To remedy the non-existence issue in the remark above, we impose an extra hypothesis relating $A_0$ to the counital subalgebras of our universally coacting weak bialgebras.
This is motivated by Claim~3 in Example~\ref{ex1}. Our main result, Theorem~\ref{thm:main-wba} below, shows that with this additional hypothesis, for any path algebra $\kk Q$, there exists a universal weak bialgebra coacting on $\kk Q$.

\begin{definition}[left UQSGd, $\Oleft(A)$;  right UQSGd, $\Oright(A)$]
\label{def:leftright-uqsgd}
Take a $\kk$-algebra $A$ as in Hypothesis~\ref{hyp:A}.
\begin{enumerate}[(a),font=\upshape]
\item \label{itm:left-uqsgd} Let $\mathcal{O}:=\Oleft(A)$ be a weak bialgebra that left coacts on $A$  with $A_0 \cong \mathcal{O}_t$ in ${}^{\mathcal{O}} \hspace{-.05in}\mathcal{A}$, so that for any weak bialgebra $H$ that left coacts on $A$ with $A_0 \cong H_t$ in $\HA,$ 
there is a unique weak bialgebra map $\pi: \mathcal{O} \to H$ so that $(\pi \otimes \id_A)\lambda^{\mathcal{O}}~=~\lambda^H.$
We refer to $\Oleft(A)$ as the {\it left universal quantum linear semigroupoid (left UQSGd) of $A$}, and refer to its coaction on $A$ as {\it universally base preserving}.

\medskip

\item \label{itm:right-uqsgd}
Let $\mathcal{O}:=\Oright(A)$ be a weak bialgebra that right coacts on $A$  with $A_0 \cong \mathcal{O}_s$ in $\mathcal{A}^{\mathcal{O}}$, so that for any weak bialgebra $H$ that right coacts on $A$ with $A_0 \cong H_s$ in $\HA,$ 
there is a unique weak bialgebra map $\pi: \mathcal{O} \to H$ so that $(\id_A \otimes \pi )\rho^{\mathcal{O}}~=~\rho^H.$
We refer to $\Oright(A)$ as the {\it right universal quantum linear semigroupoid (right UQSGd) of $A$}, and refer to its coaction on $A$ as {\it universally base preserving}.
\end{enumerate}
\smallskip

\noindent Here, the left (resp., right) $H$-coaction on $A_0$ is given by $\lambda_0$ (resp., $\rho_0$) as in Remark~\ref{rem:A0}, and the left (resp., right) $H$-coaction on $H_t$ (resp., on $H_s$) is given by $\Delta_H$ as in \cref{ex:left}.
\end{definition}

We make several remarks about the definition above. 

\begin{remark} \label{rem:A0fd-sep}
We use the assumption that $A_0$ is Frobenius and separable (from Hypothesis~\ref{hyp:A}) in the definition above and in Definition~\ref{def:manin-uqsgd} below. Namely, for any weak bialgebra $H$, the counital subalgebras $H_s$ and $H_t$ are Frobenius and separable $\kk$-algebras [Proposition~\ref{prop:Hs-facts}\ref{itm:fd-sep}]. We do not need to require that $A_0$ is commutative for Definition~\ref{def:leftright-uqsgd}.
\end{remark}

\begin{remark} \label{rem:compareuqg}
Note that the notion of universally base preserving coaction is weaker than the na\"{i}ve notion of a universal coaction discussed in Remark~\ref{rem:naive}.
Thus, the UQSGds in Definition~\ref{def:leftright-uqsgd} are more likely to exist than the universal weak bialgebras in  Remark~\ref{rem:naive}. 
 \end{remark}

\begin{remark} \label{rem:A0unitobj}
Observe that the universally base preserving condition takes a simple form when viewed through a categorical lens. By \cref{prop:alg}, we have categorical isomorphisms $\Alg(\HM) \cong \HA$ and $\Alg(\MH) \cong \AH$, and by \cref{ex:right,ex:left}, the unit objects of the monoidal categories $\Alg(\HM),\;\Alg(\MH)$ are $H_t,\; H_s$, respectively. So, the requirement that $H_t \cong A_0\;\text{in}\; \HA$ (resp., $H_s \cong A_0\; \text{in} \;\AH$) is equivalent to requiring that $A_0$ is isomorphic to the unit object of the monoidal category $\Alg(\HM)$ (resp., $\Alg(\MH)$).
\end{remark}

\begin{remark}
\label{rem:gen-Manin-oneside}
 Definition~\ref{def:leftright-uqsgd} generalizes Definition~\ref{def:UQSG-intro}, the notion of a one-sided UQSG (or, universal bialgebra that coacts from one side). Indeed, take $A$ a locally finite, connected $\mathbb{N}$-graded algebra and suppose that $\Oleft(A)$ exists. Then, $(\Oleft(A))_t = A_0 = \kk$, as $\kk$-vector spaces.
So, $\Oleft(A)$  must also be a bialgebra by Proposition~\ref{prop:Hs-facts}\ref{itm:bialgebra}, and thus, we recover the left UQSG $O^{\text{left}}(A)$ of $A$ when $A$ is connected.
\end{remark}

To generalize the transposed UQSG from Definition~\ref{def:ManinUQSG-intro}\ref{itm:ManinUQSG-intro} to the weak bialgebra setting, we need the following definitions. First, recall the transposed coaction from  Definition~\ref{def:ManinUQSG-intro}\ref{itm:transpose-intro} which we reinterpret below.

\begin{definition}
\label{def:transposecoaction}
Suppose that $H$ is a weak bialgebra coacting linearly on $A$ on the left and right via coactions $\lambda: A \to H \otimes A$ and $\rho: A \to A \otimes H$. Then for each $i$, $H$ coacts from the left and right on $A_i$ via the restrictions $\lambda_i$ and $\rho_i$. We call $A$ a {\it transposed $H$-comodule algebra}  if for each $i$, there exists a basis $\{v_j^i\}_{1\leq j\leq \dim A_i}$ for $A_i$ such that the coactions can be written in the following form:
\begin{align*}
    \lambda_i: A_i &\to H \otimes A_i & \quad  \rho_i: A_i &\to A_i \otimes H \\
    v_j^i &\mapsto \textstyle \sum_{1\leq k\leq \dim A_i} z_{j,k}^i \otimes v_k^i & \quad  v_j^i &\mapsto \textstyle \sum_{1\leq k\leq \dim A_i} v_k^i \otimes z_{k,j}^i,
\end{align*}
for some $z_{j,k}^i \in H$. 
\end{definition}

\begin{definition}[transposed UQSGd, $\Otrans(A)$]
\label{def:manin-uqsgd} 
Let $\mathcal{O}:=\Otrans(A)$ be a weak bialgebra such that $A$ is a transposed $\mathcal{O}$-comodule algebra with $A_0 \cong \mathcal{O}_t$ in ${}^{\mathcal{O}} \hspace{-.05in}\mathcal{A}$ and $A_0 \cong \mathcal{O}_s$ in $\mathcal{A}^{\mathcal{O}}$,  so that for any weak bialgebra $H$ for which $A$ is a transposed $H$-comodule algebra with  $A_0 \cong H_t$ in $\HA$ and $A_0 \cong H_s$ in $\AH$, there exists a unique weak bialgebra map  $\pi: \mathcal{O}\to H$ such that $(\pi \otimes \id_A)\lambda^{\mathcal{O}} = \lambda^H$ and $(\id_A \otimes \pi )\rho^{\mathcal{O}} = \rho^H$.  We call $\Otrans(A)$ the {\it transposed universal quantum linear semigroupoid (transposed UQSGd) of $A$}. 
\end{definition}

\begin{remark} \label{rem:A0comm}
We use the assumption that $A_0$ is commutative in Definition~\ref{def:manin-uqsgd}. Namely, by Proposition~\ref{prop:Hs-facts}\ref{itm:anti-iso}: $A_0 \cong H_t \cong H_s^{\text{op}} \cong A_0^{\text{op}}$ as $\kk$-algebras.
\end{remark}

\begin{question}[P. Etingof] \label{ques:A0comm}
Can the assumption that $A_0$ is commutative be removed by altering  Definition~\ref{def:manin-uqsgd} (so that the results in the remainder of the paper are unaffected)?
\end{question}

\begin{remark} \label{rem:gen-Manin-twoside}
For the same reasons as given in Remark~\ref{rem:gen-Manin-oneside}, we can see that the above definition is a generalization of the transposed UQSG from Definition~\ref{def:ManinUQSG-intro}\ref{itm:ManinUQSG-intro}.
\end{remark}

\begin{remark} \label{rem:uniqueness}
We only define the left/right/transposed UQSGd of $A$ up to weak bialgebra isomorphism, and it is unique (up to weak bialgebra isomorphism) if it exists.
\end{remark}

Now we show that if the left and right UQSGd of $A$ exist and are isomorphic to each other, then the transposed UQSGd of $A$ exists and is isomorphic to the left (or right) UQSGd. To proceed, consider the following terminology.

\begin{definition}
Let $H$ be a weak bialgebra and let $\lambda: A \to H \otimes A$ be a left coaction. We call this coaction \emph{inner-faithful} if, whenever $\lambda(A) \subseteq K \otimes A$ for some weak subbialgebra $K$, we must have that $K=H$. 
Right inner-faithful coactions are defined similarly.
\end{definition}

\begin{lemma} \label{lem:O-in-faith}
Take $A$ as in Hypothesis~\ref{hyp:A}, and suppose that $\Oleft(A)$ exists.
\begin{enumerate}[(a),font=\upshape]
    \item \label{itm:H-i-f} Suppose $H$ is a weak bialgebra that left coacts on $A$ with $A_0 \cong H_t$ in $\HA$. Then, $H$ coacts on $A$ inner-faithfully if and only if the weak bialgebra map $\pi: \Oleft(A) \to H$ (that arises from Definition~\ref{def:leftright-uqsgd}\ref{itm:left-uqsgd}) is surjective.
    
    \smallskip
    
    \item \label{itm:O-i-f} The weak bialgebra $\Oleft(A)$ left coacts on $A$ inner-faithfully.
\end{enumerate}
Similar statements hold for right (resp., transposed) coactions and for the UQSGd $\Oright(A)$ (resp., $\Otrans(A)$). 
\end{lemma}

\begin{proof}
\ref{itm:H-i-f} If $\pi$ is not surjective, then let $K:=\text{im}(\pi)$ which is a proper weak subbialgebra of $H$. We get that $K$ left coacts on $A$ via $\lambda^K = (\pi \otimes \id_A) \lambda^{\Oleft(A)}: A \to K \otimes A$. Therefore, $H$ does not left coact on $A$ inner-faithfully.

Conversely, suppose that $H$ does not coact on $A$ inner-faithfully, and that there exists a proper weak subbialgebra $K$ of $H$ (via inclusion $\iota$) so that the coaction of $K$ on $A$ factors through $H$ on $A$. Then,
$(\pi \otimes \id_A) \lambda^{\Oleft(A)} = \lambda^H = (\iota \otimes \id_A) \lambda^K.$
Now the $\text{im}(\pi)$ consists of the coefficients of $\lambda^K$ in $K$. So, $\text{im}(\pi)$ cannot be $H$, and $\pi$ is not surjective.

\smallskip

\ref{itm:O-i-f} This follows from part~\ref{itm:H-i-f} by taking $\pi = \id_{\Oleft(A)}$.
\end{proof}

\begin{proposition} \label{prop:left=right}
Suppose that $\Oleft(A)$ and $\Oright(A)$ exist, and let $\Ocal(A):=\Oleft(A).$ Suppose that $\Oleft(A) \cong \Oright(A)$ as weak bialgebras, and that their respective coactions on $A$ are transpose. Then $\Otrans(A)$ exists, and $\Otrans(A) \cong \mathcal{O}(A)$ as weak bialgebras.
\end{proposition}

\begin{proof}
 Assume that $\Ocal(A):=\Oleft(A)$ and $\Oright(A)$ exist. For simplicity of proof, assume that  $\Ocal(A) = \Oright(A)$ as weak bialgebras (instead of using an isomorphism). Now, suppose that we have a weak bialgebra $H$ that left coacts and right coacts  (via transposed coactions $\lambda^H: A \to H \otimes A, \;\rho^H: A \to A \otimes H$) with the property that $H_s \cong A_0$ in $\AH$ and $H_t \cong A_0$ in $\HA$. We will show that $\Ocal(A)$ satisfies the universal property described in \cref{def:manin-uqsgd}; therefore, $\Otrans(A)$ exists and $\Otrans(A) \cong \Ocal(A)$  as weak bialgebras. 

Since $\Ocal(A):=\Oleft(A)$ and $\Oright(A)$ exist, we have the following maps:
$$\lambda^L: A \to \Ocal(A) \otimes A,  \quad  \quad
\pi^L: \Ocal(A) \to H, \quad \quad
\rho^R: A \to A \otimes \Ocal(A), \quad\quad
 \pi^R: \Ocal(A) \to H.$$
with the property that $\Ocal(A)$ left coacts  on $A$ via $\lambda^L$, $\Oright(A)$ right coacts  on $A$ via $\rho^R$, $\lambda^L$ and $\rho^R$ are transposed coactions, and $\pi^L$ and $\pi^R$ are the unique weak bialgebra maps satisfying the following equations:
\begin{equation*}
    \label{eq:LRpi}
    (\pi^L \otimes \id_A)\lambda^L = \lambda^H, \qquad (\id_A \otimes \pi^R)\rho^R = \rho^H.
\end{equation*}
We make the following definitions:
\begin{equation*}
    \label{eq:deflambdarhopi}
    \lambda:= \lambda^L: A \to \Ocal(A) \otimes A, \qquad \rho:= \rho^R: A \to A \otimes \Ocal(A), \qquad \pi:=\pi^L: \Ocal(A) \to H. 
\end{equation*}
We will prove that $\pi$ is the unique weak bialgebra map such that $(\id_A \otimes \pi )\rho = \rho^H$. This will imply  that $\Ocal(A)$ has the universal property of \cref{def:manin-uqsgd}, so we must have $\Ocal(A) \cong \Otrans(A) $  as weak bialgebras. In fact, since $\pi^R$ is the unique weak bialgebra map such that $(\id_A \otimes \pi^R)\rho = \rho^H,$ it suffices to show that $\pi=\pi^R.$

Since $\lambda^H$ and $\rho^H$ are transposed coactions, for each $i$ there exists a basis $\{v_j^i\}_{1\leq j \leq \dim A_i}$ of $A_i$ such that the restricted coactions can be written in the following form:
\begin{align*}
    \lambda^H_i: A_i &\to H \otimes A_i & \quad  \rho^H_i: A_i &\to A_i \otimes H \\
    v_j^i &\mapsto \textstyle \sum_{1\leq k\leq \dim A_i} z_{j,k}^i \otimes v_k^i & \quad  v_j^i &\mapsto \textstyle \sum_{1\leq k\leq \dim A_i} v_k^i \otimes z_{k,j}^i,
\end{align*}
for some $z_{j,k}^i \in H$. Since $\{v_j^i\}$ is a basis for $A_i$ and the coactions $\lambda,\rho$ are transpose, we can write
\begin{align*}
    \lambda_i: A_i &\to \Ocal(A) \otimes A_i & \quad  \rho_i: A_i &\to A_i \otimes \Ocal(A) \\
    v_j^i &\mapsto \textstyle \sum_{1\leq k\leq \dim A_i} y_{j,k}^i \otimes v_k^i & \quad  v_j^i &\mapsto \textstyle \sum_{1\leq k\leq \dim A_i} v_k^i \otimes y_{k,j}^i,
\end{align*}
for some $y_{j,k}^i \in \Ocal(A)$.

By the previous claims, we know that $(\pi\otimes \id_A)\lambda = \lambda^H$ and $(\id_A \otimes \pi^R)\rho=\rho^H$. Therefore, for each $v_j^i$ we have
\begin{align*}
   \textstyle \sum_{1 \leq k \leq \dim A_i} \pi(y_{j,k}^i) \otimes v_k^i \;= \;(\pi\otimes \id_A)\lambda_i(v_j^i) 
    \;=\;\lambda_i^H(v_j^i) 
    \;=\;\sum_{1\leq k\leq \dim A_i} z_{j,k}^i \otimes v_k^i.
\end{align*}
Since the $\{v_k^i\}$ are a basis for $A_i$, we know that $\pi(y_{j,k}^i)=z_{j,k}^i$ for each $i,j,k.$ Similarly, for each $v_j^i$ we have
\begin{align*}
  \textstyle  \sum_{1 \leq k \leq \dim A_i} v_k^i \otimes \pi^R(y_{k,j}^i) \;=\; (\id_A \otimes \pi^R)\rho_i(v_j^i) 
    \;=\;\rho_i^H(v_j^i) 
   \; =\; \sum_{1 \leq k \leq \dim A_i} v_k^i \otimes z_{k,j}^i.
\end{align*}
Since the $\{v_k^i\}$ are a basis for $A_i$, we know that $\pi^R(y_{k,j}^i)=z_{k,j}^i$ for each $i,j,k.$ Therefore, for each $i,j,k$, we have
\[\pi^R(y_{j,k}^i)=z_{j,k}^i=\pi(y_{j,k}^i).\]
Since the coactions are inner-faithful by Lemma~\ref{lem:O-in-faith}, $\Ocal(A)$ is generated as an algebra by the $y_{j,k}^i$. (Else, there exists an algebra generator of $\Ocal(A)$ not in the set $\{y_{j,k}^i\}$ and the proper weak subbialgebra generated by the $y_{j,k}^i$ coacts on $A$, contradicting the inner-faithfulness of the coaction of $\Ocal(A)$ on $A$.) Finally, $\pi$ and $\pi^R$ are algebra maps, so we must have $\pi = \pi^R$, as desired.
\end{proof}


\section{Universal quantum linear semigroupoids of a path algebra} \label{sec:main}

In this section, we will prove our main theorem, Theorem~\ref{thm:main-wba}, constructing the left, right, and transposed UQSGds of the path algebra $\kk Q$ of a finite quiver $Q$. Furthermore, we will show that all three are isomorphic to Hayashi's face algebra $\mathfrak{H}(Q)$ (Example~\ref{ex:hay}). Note that the coactions in the following hypothesis is a specific case of that in the Definition~\ref{def:transposecoaction}. In several of our results, we will assume one of the three hypotheses given below. 

\begin{hypothesis} 
\label{hyp:coactions}
Let $Q$ be a finite quiver, and let $(H,m,u,\Delta,\ep)$ be a weak bialgebra. Consider the following formulas for each $j \in Q_0$ and each $q \in Q_1$,
\begin{align*}
\lambda: \kk Q &\to H \otimes \kk Q & \rho: \kk Q &\to \kk Q \otimes H \\
 e_j &\mapsto \textstyle \sum_{i \in Q_0} y_{j,i} \otimes e_i & e_j &\mapsto \textstyle \sum_{i \in Q_0}  e_i \otimes y_{i,j}\\
 q &\mapsto \textstyle \sum_{p \in Q_1} y_{q,p} \otimes p  & q &\mapsto \textstyle \sum_{p \in Q_1}  p\otimes y_{p,q}
\end{align*}
for some elements $y_{i,j} \in H_0$ and $y_{p,q} \in H_1$. We will consider three separate hypotheses in the sequel.
\begin{enumerate}[(a),font=\upshape]
    \item Assume that $\kk Q$ is a left $H$-comodule algebra via a coaction of the form $\lambda$. \label{hyp:coactions-L}
    \item \label{hyp:coactions-R}Assume that $\kk Q$ is a right $H$-comodule algebra via a coaction of the form $\rho$.
    \item \label{hyp:coactions-Manin} Assume that $\kk Q$ is a transposed $H$-comodule algebra via $\lambda$ and  $\rho$.
\end{enumerate}
\end{hypothesis}

The following results will be of use in this section.

\begin{lemma} \label{lem:wba-structure}
Let $H$ be a weak bialgebra which coacts on $\kk Q$ as in \cref{hyp:coactions}. Consider the following formulas for any $i,j,k\in Q_0$ and $p,q \in Q_1$:
\begin{align}
\label{eq:coalgreln1}
\textstyle \Delta(y_{i,j}) = \sum_{k \in Q_0} y_{i,k} \otimes y_{k,j}, \quad & \quad \varepsilon(y_{i,j}) = \delta_{i,j} \cdot 1_\kk, \\
\label{eq:coalgreln2} \textstyle
\Delta(y_{p,q}) = \sum_{r \in Q_1} y_{p,r} \otimes y_{r,q},  \quad &\quad \varepsilon(y_{p,q}) = \delta_{p,q} \cdot 1_\kk,\\
\label{eq:algreln1}
y_{k,i}y_{k,j}&=\delta_{i,j}y_{k,i} \\
\label{eq:algreln2}
y_{i,k}y_{j,k}&=\delta_{i,j}y_{i,k} \\
\label{eq:algreln3}
y_{s(p),s(q)}y_{p,q}&=y_{p,q} \\
\label{eq:algreln4}
y_{p,q}y_{t(p),t(q)}&=y_{p,q}.
\end{align}
\begin{enumerate}[(a),font=\upshape]
    \item \label{lem:wba-structure-a} If \cref{hyp:coactions}\ref{hyp:coactions-L} holds, then $H$ satisfies \eqref{eq:coalgreln1}, \eqref{eq:coalgreln2}, \eqref{eq:algreln2}, \eqref{eq:algreln3}, and \eqref{eq:algreln4}.
    \item \label{lem:wba-structure-b} If \cref{hyp:coactions}\ref{hyp:coactions-R} holds, then $H$ satisfies \eqref{eq:coalgreln1}, \eqref{eq:coalgreln2}, \eqref{eq:algreln1}, \eqref{eq:algreln3}, and \eqref{eq:algreln4}.
    \item \label{lem:wba-structure-c} If \cref{hyp:coactions}\ref{hyp:coactions-Manin} holds, then $H$ satisfies \eqref{eq:coalgreln1} to \eqref{eq:algreln4}.
\end{enumerate}
\end{lemma}

\begin{proof}
We will prove \ref{lem:wba-structure-a}. The proof for \ref{lem:wba-structure-b} is similar and hence omitted, while \ref{lem:wba-structure-c} follows from \ref{lem:wba-structure-a} and \ref{lem:wba-structure-b}.

The formulas for $\Delta$ and $\varepsilon$ follow from the coassociativity and counitality of $\lambda$. For example, for $i \in Q_0$,
\[
\begin{array}{rl}
\medskip
\textstyle
\sum_{j \in Q_0}\Delta(y_{i,j}) \otimes e_j
&= (\Delta \otimes \id)\lambda(e_i) 
= (\id \otimes \lambda)\lambda(e_i)\\
&= \sum_{k \in Q_0}  y_{i,k} \otimes \lambda(e_k)
= \sum_{j,k \in Q_0}  y_{i,k} \otimes y_{k,j} \otimes e_j.
\end{array}
\]
Since the $e_j$ are linearly independent, we have that $\Delta(y_{i,j}) = \sum_{k \in Q_0} y_{i,k} \otimes y_{k,j}$.
Further, since $e_i = \id_{\kk Q}(e_i) = (\varepsilon \otimes \id_{\kk Q})\lambda(e_i) = \textstyle \sum_{j\in Q_0} \varepsilon(y_{i,j}) e_j$, we conclude that $\varepsilon(y_{i,j}) = \delta_{i,j} \cdot 1_\kk$. This proves \eqref{eq:coalgreln1}; the proof for \eqref{eq:coalgreln2} is similar.

Next, we will use the fact that $\kk Q$ is a left $H$-comodule algebra. We have
\begin{align*}
 \textstyle\sum_{k\in Q_0}\delta_{i,j}y_{i,k}\otimes e_k &=
\lambda(\delta_{i,j}e_i) \\
&=\lambda(e_ie_j)\\ &\overset{\eqref{def:Halg-L-mult}}{=}\lambda(e_i)\lambda(e_j) \\
&=\l(\textstyle \sum_{k\in Q_0} y_{i,k} \otimes e_k\r)\l(\textstyle \sum_{\ell\in Q_0}y_{j,\ell} \otimes e_{\ell}\r) \\
&=\textstyle \sum_{k,{\ell} \in Q_0} y_{i,k}y_{j,{\ell}} \otimes e_ke_{\ell} \\
&= \textstyle \sum_{k\in Q_0} y_{i,k}y_{j,k} \otimes e_k.
\end{align*}
Since the $e_k$ are linearly independent, we must have $\delta_{i,j}y_{i,k}=y_{i,k}y_{j,k},$ that is, \eqref{eq:algreln2} holds.

To show \eqref{eq:algreln3}, notice that for $p\in Q_1$ we have
\begin{align*}
\textstyle \sum_{q\in Q_1} y_{p,q}\otimes q
&= \lambda(p) \\
&=\lambda(e_{s(p)}p) \\
&= \lambda(e_{s(p)})\lambda(p) \\
&=\l(\textstyle \sum_{i \in Q_0} y_{s(p),i}\otimes e_i\r)\l(\textstyle \sum_{q\in Q_1} y_{p,q}\otimes q\r) \\
&= \textstyle \sum_{i \in Q_0, q \in Q_1} y_{s(p),i}y_{p,q} \otimes e_iq \\
&= \textstyle \sum_{q\in Q_1} y_{s(p),s(q)}y_{p,q} \otimes q.
\end{align*}
By linear independence of the set $\{q\}_{q \in Q_1}$, we have $y_{p,q}=y_{s(p),s(q)}y_{p,q}$ for all $p,q \in Q_1$.

 We use the relation $p=pe_{t(p)}$ for $p\in Q_1$ to prove \eqref{eq:algreln4} in the same manner as \eqref{eq:algreln3}.
\end{proof}

The following proposition is a collection of identities that hold if we only assume the existence of a left $H$-coaction making $\kk Q$ an $H$-comodule algebra.

\begin{proposition} \label{prop:wba-idempotent-L}
Let $H$ be a weak bialgebra which coacts on $\kk Q$ as in \cref{hyp:coactions}\ref{hyp:coactions-L}. For each $j \in Q_0$, consider the elements
$$
\eta_j:=\textstyle \sum_{i \in Q_0} y_{i,j}, \qquad \theta_j:=\textstyle \sum_{i \in Q_0} y_{j,i}.
$$
The following statements hold. 
\begin{enumerate}[(a),font=\upshape]
\item  \label{itm:nonzero}
For each $j \in Q_0$, $\eta_j$ and $\theta_j$ are non-zero elements of $H$.
\item \label{itm:idem-L} For each $j \in Q_0$, $\eta_j$ is an idempotent element of $H_s$.
\item \label{itm:psi-L}
If $\kk Q_0\cong_{\psi} H_t$ as left $H$-comodule algebras, then the following statements hold: 
    \begin{enumerate}[(i), font=\upshape]
        \item \label{eq:eta-basis} $\{\eta_k\}_{k \in Q_0}$ is a $\kk$-basis of $H_s$.
        \item \label{eq:identity-L} $1_H=\sum_{i,j\in Q_0} y_{i,j}.$
        \item \label{eq:psiek=theta-L} For each $k \in Q_0$, $\psi(e_k)=\theta_k$; hence $\theta_k \in H_t$.
        \item \label{eq:thetaorthidem-L} The set $\{\theta_i\}_{i\in Q_0}$ is a $\kk$-basis for $H_t$ of orthogonal idempotent elements.
        \item
        \label{eq:etaorthidem-L} The set  $\{\eta_j\}_{j\in Q_0}$ is a $\kk$-basis for $H_s$ of orthogonal idempotent elements.
        \item \label{eq:crossitem-L} For all $i,j,k,{\ell} \in Q_0$, $y_{i,j}y_{k,{\ell}}=\delta_{i,k}\delta_{j,{\ell}}y_{i,j}$.
    \end{enumerate}
\end{enumerate}
\end{proposition}

\begin{proof} 
\ref{itm:nonzero}
To show that $\eta_j$ is non-zero, we note that
\begin{equation*} 
\label{eq:epeta-L}
\textstyle \ep(\eta_j)=\ep\l(\sum_{i \in Q_0} y_{i,j}\r)\overset{\eqref{eq:coalgreln1}}{=}1.
\end{equation*}
A similar calculation shows that $\ep(\theta_j)=1$, so $\theta_j$ is non-zero.

\medskip

\ref{itm:idem-L} Let $j \in Q_0$. Then
\[\eta_j^2 = 
\textstyle \sum_{i,k \in Q_0} y_{i,j}y_{k,j} \overset{\eqref{eq:algreln2}}{=}  \sum_{i \in Q_0} y_{i,j} = \eta_j\]
so $\eta_j$ is idempotent. Moreover, note that 
\[
\begin{array}{rl}
\medskip
\textstyle \sum_{i,j \in Q_0} y_{i,j} \otimes e_j
&= \lambda(\sum_{i \in Q_0} e_i)
\; \; = \lambda(1_{\kk Q})
\; \;\overset{\eqref{def:Halg-L-unit}}{=} (\varepsilon_s \otimes \id_{\kk Q})\lambda(1_{\kk Q})\\
& = (\varepsilon_s \otimes \id_{\kk Q})\lambda(\sum_{i \in Q_0} e_i)
\; \; = \sum_{j \in Q_0} \varepsilon_s(\sum_{i \in Q_0} y_{i,j}) \otimes e_j.
\end{array}
\]
Since the $e_j$ are linearly independent, for each $j \in Q_0$ we have $\eta_j:=\textstyle \sum_{i \in Q_0} y_{i,j} \in H_s$. 

\medskip

\ref{itm:psi-L}  Suppose that $\kk Q_0 \cong H_t$ as left $H$-comodule algebras. Then there exists an algebra isomorphism $\psi: \kk Q_0 \to H_t$ which is also a map of left $H$-comodules. Hence, 
\begin{equation}
\label{eq:comod-diagram-L}
(\id_H \otimes \psi)\lambda_0 = \lambda_{H_t} \psi,
\end{equation}
for $\lambda_{H_t} = \Delta_H|_{H_t}$ by Example~\ref{ex:left}.

\smallskip

\ref{eq:eta-basis} Evaluating the left-hand side on $1_{\kk Q_0}$, we have
\begin{align*}(\id \otimes \psi) \lambda_0( 1_{\kk Q}) &= \textstyle (\id \otimes \psi) \lambda_0(\sum_{i\in Q_0} e_i) = (\id \otimes \psi)(\sum_{i,j \in Q_0} y_{i,j} \otimes e_j) \\&= \textstyle \sum_{j\in Q_0} \left(\sum_{i\in Q_0} y_{i,j}\right) \otimes \psi(e_j)  = \sum_{j\in Q_0} \eta_j \otimes \psi(e_j).
\end{align*}
On the other hand, $\psi$ is an algebra map, so $\psi(1_{\kk Q}) = 1_H$. Thus, $\Delta \psi(1_{\kk Q}) = 1_1 \otimes 1_2$. Hence,
\begin{equation} \label{eq:L1}
    \textstyle 1_1 \otimes 1_2 = \sum_{j \in Q_0} \eta_j \otimes \psi(e_j).
\end{equation}
Since the distinct $e_j$ are linearly independent and $\psi$ is an algebra isomorphism, the $\psi(e_j)$ are also linearly independent. Now by Proposition~\ref{prop:Hs-facts}\ref{itm:coideal}, we conclude that the $\{\eta_j\}_{j \in Q_0}$ span $H_s$. By \cref{prop:Hs-facts}\ref{itm:anti-iso}, we have $\dim_{\kk}H_t=\dim_\kk H_s$. Therefore $\dim_\kk H_s = \dim_\kk \kk Q_0 = |Q_0|$, so we have that $\{\eta_j\}_{j \in Q_0}$ is a $\kk$-basis of $H_s$.

\smallskip

\ref{eq:identity-L} For any $k \in Q_0$ we have
\[
\begin{array}{rll}
\textstyle \sum_{j\in Q_0} \eta_j\psi(e_k)\otimes \psi(e_j) &\overset{\eqref{eq:L1}}{=} 1_1 \psi(e_k) \otimes  1_2 &\overset{\textnormal{\ref{prop:Hs-facts}}\textnormal{\ref{itm:deltaHsht}}}{=} \Delta \psi(e_k) \\
&\overset{\eqref{eq:comod-diagram-L}}{=} (\id_H \otimes \psi)\lambda_0(e_k) &= \sum_{j\in Q_0} y_{k,j}\otimes \psi(e_j).
\end{array}
\]
Since the $\psi(e_j)$ are linearly independent, we have that
\begin{equation} 
\label{eq:L2}
\eta_j\psi(e_k) =y_{k,j}
\end{equation}
for each $j,k \in Q_0$. 

 By \eqref{eq:L1},
\begin{equation}\label{eq:L3}
1_H = 1_1 \ep(1_2) = \textstyle \sum_{j \in Q_0} \eta_j \ep(\psi(e_j)).
\end{equation}

Next, consider the following calculation:
\begin{align*}
    \textstyle \sum_{k\in Q_0} \eta_k \otimes \psi(e_k) &\overset{\eqref{eq:L1}}{=}1_1 \otimes 1_2 \\
    &=\Delta(1_H) \\
    &\overset{\eqref{eq:L3}}{=}\Delta(\textstyle \sum_{j\in Q_0} \ep(\psi(e_j)) \eta_j) \\
    &=\Delta(\textstyle \sum_{i,j \in Q_0} \ep(\psi(e_j)) y_{i,j}) \\
    &\overset{\eqref{eq:coalgreln1}}{=} \textstyle \sum_{i,j,k \in Q_0} \ep(\psi(e_j)) y_{i,k}\otimes y_{k,j} \\
    & = \textstyle \sum_{k \in Q_0} (\sum_{i \in Q_0} y_{i,k}) \otimes  (\sum_{j\in Q_0}\ep(\psi(e_j)) y_{k,j})\\
    & = \textstyle \sum_{k \in Q_0} \eta_k \otimes  (\sum_{j\in Q_0}\ep(\psi(e_j)) y_{k,j}).
\end{align*}
Since, by \ref{eq:eta-basis}, the $\{\eta_j\}_{j\in Q_0}$ are linearly independent, we must have \begin{equation} \label{eq:L4}
    \psi(e_k)=\textstyle \sum_{j\in Q_0}\ep(\psi(e_j)) y_{k,j}
\end{equation} 
for each $k\in Q_0.$ Notice that
\begin{equation*}
  \textstyle  y_{k,j} \overset{\eqref{eq:L2}}{=} \eta_j \psi(e_k) 
    \overset{\eqref{eq:L4}}{=} \eta_j  \sum_{{\ell}\in Q_0} \ep(\psi(e_{\ell})) y_{k,{\ell}}
    = \sum_{i,{\ell} \in Q_0} \ep(\psi(e_{\ell})) y_{i,j}y_{k,{\ell}}.
\end{equation*}
Multiplying both sides of the equation on the left by $y_{k,j}$ yields
\begin{equation} \label{eq:L5}
\begin{array}{rl}
\smallskip
    y_{k,j} &\overset{\eqref{eq:algreln2}}{=} (y_{k,j})^2 \\
    \smallskip
    &=\sum_{i,{\ell} \in Q_0} \ep(\psi(e_{\ell})) y_{k,j}y_{i,j}y_{k,{\ell}}\\
    \smallskip
    &\overset{\eqref{eq:algreln2}}{=} \sum_{i,{\ell} \in Q_0} \ep(\psi(e_{\ell})) \delta_{i,k}y_{k,j}y_{k,{\ell}}\\
    \smallskip
    &\overset{}{=} y_{k,j}\sum_{{\ell} \in Q_0} \ep(\psi(e_{\ell})) y_{k,{\ell}}\\
    \smallskip
    &\overset{\eqref{eq:L4}}{=} y_{k,j}\psi(e_k).
\end{array}
\end{equation}

Now for each $k\in Q_0$, we get 
\begin{align*}
    1_\kk&\overset{\eqref{eq:coalgreln1}}{=}\ep(y_{k,k}) \\
    &\overset{\eqref{eq:L5}}{=} \ep(y_{k,k}\psi(e_k)) \\
    &\overset{\textnormal{\ref{def:wba}}\textnormal{\ref{def:wba4}}}{=}\ep(y_{k,k}1_1)\ep(1_2\psi(e_k)) \\
    &\textstyle \overset{\eqref{eq:L1}}{=}\sum_{i\in Q_0} \ep(y_{k,k}\eta_i)\ep(\psi(e_i)\psi(e_k)) \\
    &\textstyle \overset{\psi\text{ alg. map}}{=} \sum_{i\in Q_0}\ep(y_{k,k}\eta_i)\ep(\delta_{i,k}\psi(e_k))\\
    &\overset{}{=} \ep(y_{k,k}\eta_k)\ep(\psi(e_k))\\
   &\overset{}{=} \textstyle \ep\l(\sum_{i\in Q_0} y_{k,k}y_{i,k}\r)\ep(\psi(e_k)) \\
    &\textstyle \overset{\eqref{eq:algreln2}}{=} \ep\l(\sum_{i\in Q_0} \delta_{i,k}y_{k,k}\r)\ep(\psi(e_k)) \\
    &\overset{\eqref{eq:coalgreln1}}{=}\ep(\psi(e_k)).
\end{align*}
Finally,
\[
\textstyle
1_H \overset{\eqref{eq:L3}}{=}\sum_{j \in Q_0} \ep(\psi(e_j))\eta_j  =\sum_{j\in Q_0}\eta_j = \sum_{i,j\in Q_0}y_{i,j}.
\]

\medskip

\ref{eq:psiek=theta-L} For each $k \in Q_0$, we have 
\begin{equation*} \label{eq:L6}
\textstyle \theta_k = \sum_{j \in Q_0} y_{k,j} \overset{\eqref{eq:L2}}{=} 
\sum_{j \in Q_0} \eta_j \psi(e_k)
 = (\sum_{i,j \in Q_0} y_{i,j}) \psi(e_k)
 \overset{\textnormal{\ref{eq:identity-L}}}{=}\psi(e_k).
\end{equation*}

\medskip

\ref{eq:thetaorthidem-L} Since $\{e_i\}_{i \in Q_0}$ is a $\kk$-basis of $\kk Q_0$ of orthogonal idempotent elements and $\psi$ is an algebra isomorphism, $\{\psi(e_i)\}_{i \in Q_0}$ is a $\kk$-basis of orthogonal idempotents of $H_t$. By, part \ref{eq:psiek=theta-L}, the claim follows. 

\medskip

\ref{eq:etaorthidem-L} Since $\kk Q_0 \cong_{\psi} H_t$ as algebras, by Proposition~\ref{prop:Hs-facts}\ref{itm:anti-iso} we have $H_s\cong_{\gamma}\bigoplus_{i\in Q_0}\kk$. Let $\{E_i\}_{i\in Q_0}$ be a set of primitive idempotents of $\bigoplus_{i\in Q_0}\kk$. 
By part \ref{itm:idem-L}, for each $k \in Q_0$, $\eta_k$ is an idempotent of $H_s$.
Hence, $\gamma(\eta_k)=\sum_{i\in I_k}E_i$ for some subset $I_k$ of $Q_0$. Therefore, 
\[\textstyle \sum_{k\in Q_0}\sum_{i\in I_k}E_i=\sum_{k\in Q_0}\gamma(\eta_k)=\gamma(\sum_{i,k\in Q_0}y_{i,k}) \overset{\textnormal{\ref{eq:identity-L}}}{=} \gamma(1_H)=1_{\bigoplus \kk}=\sum_{i\in Q_0}E_i.\]
As a result, we conclude that $\gamma(\eta_k)=E_i$ for some $i\in Q_0$ and for $k \neq j$, we have that $\gamma(\eta_k)\ne \gamma(\eta_j)$.
Thus, $\{\eta_k\}_{k\in Q_0}$ is a set of orthogonal idempotents of $H_s$.

\medskip

\ref{eq:crossitem-L} 
For each $i,j,k,{\ell} \in Q_0$, we have:
$$
    y_{i,j}y_{k,{\ell}}
    \overset{\eqref{eq:L2}}{=}\eta_j\psi(e_i)\eta_{\ell}\psi(e_k)
    \overset{\textnormal{\ref{eq:psiek=theta-L}}}{=}
    \eta_j \theta_i \eta_{\ell} \theta_k
    =
    \theta_i  \theta_k \eta_j \eta_{\ell}
    \overset{\textnormal{\ref{eq:thetaorthidem-L}}, \textnormal{\ref{eq:etaorthidem-L}}}{=}\delta_{i,k} \delta_{j,{\ell}} \theta_i \eta_j
    \overset{\textnormal{\ref{eq:psiek=theta-L}},\eqref{eq:L2}}{=}\delta_{i,k}\delta_{j,{\ell}}y_{i,j}, 
$$
where the third equality holds by parts \ref{itm:idem-L}, \ref{eq:psiek=theta-L}, and Proposition~\ref{prop:Hs-facts}\ref{itm:HsHtcommute}.
\end{proof}

\medskip

The analogue of Proposition~\ref{prop:wba-idempotent-L} for a weak bialgebra coaction on $\kk Q$ satisfying \cref{hyp:coactions}\ref{hyp:coactions-R} also holds, and follows by a similar proof.

\begin{proposition} \label{prop:wba-idempotent-R}
Let $H$ be a weak bialgebra which coacts on $\kk Q$ as in \cref{hyp:coactions}\ref{hyp:coactions-R}. For each $j \in Q_0$, consider the elements
$$
\eta_j:=\textstyle \sum_{i \in Q_0} y_{i,j}, \qquad \theta_j:=\textstyle \sum_{i \in Q_0} y_{j,i}.
$$
The following statements hold. 
\begin{enumerate}[(a),font=\upshape]
\item  \label{itm:nonzero-R}
For each $j \in Q_0$, $\eta_j$ and $\theta_j$ are non-zero elements of $H$.
\item \label{itm:idem-R} For each $j \in Q_0$, $\theta_j$ is an idempotent element of $H_t$.
\item \label{itm:phi-R}
If $\kk Q_0\cong_{\phi} H_s$ as right $H$-comodule algebras, then the following statements hold: 
    \begin{enumerate}[(i), font=\upshape]
        \item \label{eq:theta-basis} $\{\theta_k\}_{k \in Q_0}$ is a $\kk$-basis of $H_t$.
        \item \label{eq:identity-R} $1_H=\sum_{i,j\in Q_0} y_{i,j}.$
        \item \label{eq:phiek=theta-R} For each $k \in Q_0$, $\phi(e_k)=\eta_k$; hence $\eta_k \in H_s$.
        \item
        \label{eq:etaorthidem-R} The set  $\{\eta_j\}_{j\in Q_0}$ is a $\kk$-basis for $H_s$ of orthogonal idempotent elements.
        \item \label{eq:thetaorthidem-R} The set $\{\theta_i\}_{i\in Q_0}$ is a $\kk$-basis for $H_t$ of orthogonal idempotent elements.
        \item \label{eq:crossitem-R} For all $i,j,k,{\ell} \in Q_0$, $y_{i,j}y_{k,{\ell}}=\delta_{i,k}\delta_{j,{\ell}}y_{i,j}$. \qed
    \end{enumerate}
\end{enumerate} 
\end{proposition}

This brings us to the main result of the paper.

\begin{theorem} \label{thm:main-wba}
Let $Q$ be a finite quiver with path algebra $\kk Q$. Then the universal quantum linear semigroupoids $\Oleft(\kk Q)$, $\Oright(\kk Q)$, and $\Otrans(\kk Q)$ exist, and they are each isomorphic to $\mathfrak{H}(Q)$ as weak bialgebras.
\end{theorem}

\begin{proof}
Consider the left and right coaction of $\hay$ on $\kk Q$ presented in Example~\ref{ex:haycoaction}. We will show in full detail that $\Oleft(\kk Q) \cong \hay$ as weak bialgebras (under Hypothesis~\ref{hyp:coactions}\ref{hyp:coactions-L}), and briefly discuss the proof that $\Oright(\kk Q) \cong \hay$ as weak bialgebras  (under Hypothesis~\ref{hyp:coactions}\ref{hyp:coactions-R}). 
 Then we have that $\Oleft(\kk Q) \cong \hay \cong \Oright(\kk Q)$ as weak bialgebras, and that the coactions of the left/right UQSGds are transposed via Example~\ref{ex:haycoaction}. 
Hence, Proposition~\ref{prop:left=right} yields  $\Otrans(\kk Q) \cong \hay$ as weak bialgebras  (under Hypothesis~\ref{hyp:coactions}\ref{hyp:coactions-Manin}).

\smallskip

 To proceed, we will show that $\hay$ satisfies the universal property of $\Oleft(\kk Q)$. Indeed we have that $\hay$ is a weak bialgebra that left coacts on $\kk Q$ [Example~\ref{ex:haycoaction}]. Moreover,  $\kk Q_0 \cong (\hay)_t$ as left $\hay$-comodule algebras: the algebra isomorphism, call it $\tau$, holds by Proposition~\ref{prop:epHQ} via $e_i \mapsto a_i'$, which is a left comodule map due to the computation below:
\begin{align*}
\lambda_{\hay_t} \tau (e_i) 
&= \textstyle \lambda_{\hay_t} (a_i')\\
&\overset{\textnormal{\ref{prop:epHQ}},~\textnormal{\ref{ex:left}}}{=} \textstyle \sum_{k\in Q_0} \Delta_{\hay} (x_{i,k})\\
&\overset{\eqref{eq:coalgHQ}}{=} \textstyle \sum_{j,k\in Q_0} x_{i,j} \otimes x_{j,k}\\
&\overset{\textnormal{\ref{prop:epHQ}}}{=}  \textstyle \sum_{j \in Q_0} (\id_{\hay} \otimes \tau) (x_{i,j} \otimes e_j)\\  
&\overset{\textnormal{\ref{ex:haycoaction}}}{=} \textstyle (\id_{\hay} \otimes \tau) \lambda_{\kk Q_0}(e_i).
\end{align*}

 Now, assume that $H$ is a weak bialgebra which coacts from the left on $\kk Q$ as in Hypothesis~\ref{hyp:coactions}\ref{hyp:coactions-L},  recall Remark~\ref{rem:A0}, and 
assume that there exists an isomorphism
$$\psi: \kk Q_0 \overset{\sim}{\longrightarrow} H_t \quad \text{in } \HA.$$ 
 Recall that we have elements $\{y_{i,j}\}_{i,j \in Q_0}$ and $\{y_{p,q}\}_{p,q \in Q_1}$ in $H$, as well as idempotents $\{\eta_i\}_{i \in Q_0}$ in $H_s$ and $\{\theta_i\}_{i \in Q_0}$ in $H_t$, as in Lemma~\ref{lem:wba-structure} and \cref{prop:wba-idempotent-L}. 
Now consider the map $\pi$ defined on the algebra generators of $\hay$ and extended multiplicatively and linearly:
$$\pi: \hay \to H \quad \text{defined by}  \quad
x_{i,j} \mapsto y_{i,j} \text{ for } i,j \in Q_0, \quad
x_{p,q} \mapsto y_{p,q} \text{ for } p,q \in Q_1.
$$
We aim to show first that $\pi$ is a weak bialgebra map (i.e., that $\pi$ is an algebra map and a coalgebra map) satisfying $(\pi \otimes \id_{\kk Q})\lambda^{\hay} = \lambda^H$, and that $\pi$ is the only such weak bialgebra map $\hay \to H$ with this property. This would achieve the result that $\Oleft(\kk Q) \cong \hay$ as weak bialgebras.

To show that $(\pi \otimes \id_{\kk Q})\lambda^{\hay} = \lambda^H$, note that for $i\in Q_0$,
\[(\pi \otimes \id_{\kk Q}) \lambda^{\hay}(e_i)  = (\pi \otimes \id_{\kk Q})\left(\textstyle \sum_{j \in Q_0} x_{i,j} \otimes e_j\right)
= \textstyle \sum_{j \in Q_0}y_{i,j} \otimes e_j = \lambda^H(e_i).
\]
A similar calculation shows that $(\pi \otimes \id_{\kk Q})\lambda^{\hay}(p) = \lambda^H(p)$ for $p \in Q_1$. Since $\pi, \:\lambda^{\hay}$ and $\lambda^H$ are multiplicative, we must have $(\pi \otimes \id_{\kk Q})\lambda^{\hay} = \lambda^H$.

\smallskip

The unitality of $\pi$ follows from the computation:
$$\textstyle \pi(1_{\hay}) 
\overset{\eqref{eq:unitHQ}}{=} \pi(\sum_{i,j \in Q_0} x_{i,j}) 
=
\sum_{i,j \in Q_0} y_{i,j} 
\overset{\textnormal{\ref{prop:wba-idempotent-L}}\textnormal{\ref{itm:psi-L}\textnormal{\ref{eq:identity-L}}}}{=} 1_H.$$
To prove that $\pi$ is multiplicative, note that by Proposition~\ref{prop:wba-idempotent-L}\ref{itm:psi-L}\ref{eq:crossitem-L}, for all $i,j,k,l \in Q_0$,
$$y_{i,j} y_{k,l} = \delta_{i,k} \delta_{j,l} y_{i,j}.$$
By \eqref{eq:algreln3} and \eqref{eq:algreln4} in \cref{lem:wba-structure}, we have 
\[y_{s(p),s(q)}y_{p,q}=y_{p,q}=y_{p,q}y_{t(q),t(p)}.\]
So, we obtain that
\begin{align} \label{eq:mult-H}
\begin{array}{rl}
\smallskip
y_{p,q}y_{p',q'}
&\overset{\eqref{eq:algreln3},\eqref{eq:algreln4}}{=} y_{p,q} \; y_{t(p),t(q)} \;y_{s(p'),s(q')}\;  y_{p',q'}\\
\smallskip
&=  \delta_{t(p),s(p')} \; \delta_{t(q),s(q')}\; y_{p,q}\; y_{t(p),t(q)} \; y_{p',q'}\\
&\overset{\eqref{eq:algreln4}}{=}  \delta_{t(p),s(p')} \; \delta_{t(q),s(q')}\; y_{p,q} \; y_{p',q'}.
\end{array}
\end{align}
Now  \eqref{eq:multHQ1}, \eqref{eq:multHQ01} and \eqref{eq:multHQ0} imply that $\pi$ is multiplicative. Therefore, $\pi$ is an algebra map.
 
\smallskip

Next, we will show that 
$\pi$ is also a coalgebra map, i.e., that $\Delta_H\; \pi = (\pi \otimes \pi) \Delta_{\hay}$ and $\ep_H \; \pi = \ep_{\hay}$.
We will prove this for $x_{p,q}$ by induction on the length $\ell$ of the paths $p,q \in Q$.
If $\ell = 0, 1$, then the assertion holds by \eqref{eq:coalgreln1} and \eqref{eq:coalgreln2} in Lemma~\ref{lem:wba-structure}.
Now take 
\[ p=p_1 \cdots p_{\ell-1}p_\ell \quad \text{ and } \quad q=q_1 \cdots q_{\ell-1}q_\ell
\]
paths of length~$\ell$ with $p_i,q_i \in Q_1$. 
Then, for $\ell \geq 2$:
\begin{align*}
\medskip
&\Delta_H \; \pi(x_{p_1,q_1} \cdots x_{p_{\ell-1},q_{\ell-1}}x_{p_{\ell},q_{\ell}}) \\
&= \Delta_H(y_{p_1,q_1} \cdots y_{p_{\ell-1},q_{\ell-1}}y_{p_{\ell},q_{\ell}})\\
\medskip
&= \Delta_H(y_{p_1,q_1} \cdots y_{p_{\ell-1},q_{\ell-1}})\; \Delta_H(y_{p_{\ell}, q_{\ell}})\\
\medskip
&= (\Delta_H \; \pi)(x_{p_1,q_1} \cdots x_{p_{\ell-1},q_{\ell-1}})\; (\Delta_H \; \pi)(x_{p_{\ell}, q_{\ell}})\\
\medskip
& \overset{\text{induction}}{=}
(\pi \otimes \pi)\Delta_{\hay}(x_{p_1,q_1} \cdots x_{p_{\ell-1},q_{\ell-1}})\;
(\pi \otimes \pi) \Delta_{\hay}(x_{p_{\ell},q_{\ell}})\\
\medskip
& =(\pi \otimes \pi)  \Delta_{\hay}(x_{p_1,q_1} \cdots x_{p_{\ell-1},q_{\ell-1}}x_{p_{\ell},q_{\ell}}),
\end{align*}
where the first three equalities and the last equality hold because $\pi$, $\Delta_H$, and $\Delta_{\hay}$ preserve multiplication.
Further, we have
\begin{align*}
&(\ep_H \; \pi)(x_{p_1,q_1}\cdots x_{p_{\ell-1},q_{\ell-1}} x_{p_{\ell},q_{\ell}})\\ 
&= \ep_H(y_{p_1,q_1}\cdots y_{p_{\ell-1},q_{\ell-1}} y_{p_{\ell},q_{\ell}}) \\
&= \ep_H(y_{p_1,q_1}\cdots y_{p_{\ell-1},q_{\ell-1}}1_H y_{p_{\ell},q_{\ell}}) \\
&\overset{\textnormal{\ref{def:wba}}\textnormal{\ref{def:wba4}}}{=} \ep_H(y_{p_1,q_1}\cdots y_{p_{\ell-1},q_{\ell-1}}1_1) \; \ep_H(1_2 y_{p_{\ell},q_{\ell}})\\
&\overset{\textnormal{\ref{prop:wba-idempotent-L}}\textnormal{\ref{itm:psi-L}\textnormal{\ref{eq:identity-L}}},\eqref{eq:coalgreln1}}{=} \textstyle \sum_{i,j,k \in Q_0}\ep_H(y_{p_1,q_1}\cdots y_{p_{\ell-1},q_{\ell-1}}y_{i,k}) \; \ep_H(y_{k,j} y_{p_{\ell},q_{\ell}})\\
&\overset{\eqref{eq:mult-H}}{=} \textstyle \sum_{i,j,k \in Q_0}\ep_H(y_{p_1,q_1}\cdots y_{p_{\ell-1},q_{\ell-1}}\delta_{i,t(p_{\ell-1})}\delta_{k,t(q_{\ell-1})})\; \ep_H(\delta_{k,s(p_{\ell})}\delta_{j,s(q_{\ell})} y_{p_{\ell},q_{\ell}})\\
& = \textstyle \sum_{k \in Q_0}\delta_{k,t(q_{\ell-1})}\delta_{k,s(p_{\ell})}\ep_H(y_{p_1,q_1}\cdots y_{p_{\ell-1},q_{\ell-1}})\; \ep_H( y_{p_{\ell},q_{\ell}})\\
&=\delta_{t(q_{\ell-1}),s(p_{\ell})} \ep_H(y_{p_1,q_1}\cdots y_{p_{\ell-1},q_{\ell-1}})\ep_H( y_{p_{\ell},q_{\ell}})\\
&=\delta_{t(q_{\ell-1}),s(p_{\ell})} (\ep_H \; \pi)(x_{p_1,q_1}\cdots x_{p_{\ell-1},q_{\ell-1}})\; (\ep_H \;\pi)( x_{p_{\ell},q_{\ell}})\\
&\overset{\text{induction}}{=}\delta_{t(q_{\ell-1}),s(p_{\ell})} \ep_{\hay}(x_{p_1,q_1}\cdots x_{p_{\ell-1},q_{\ell-1}})\;  \ep_{\hay}( x_{p_{\ell},q_{\ell}})\\
&=\delta_{t(q_{\ell-1}),s(p_{\ell})} \ep_{\hay}(x_{p_1,q_1}\cdots x_{p_{\ell-1},q_{\ell-1}})\delta_{p_{\ell},q_{\ell}}\\
&=\delta_{t(q_{\ell-1}),s(q_{\ell})} \ep_{\hay}(x_{p_1,q_1}\cdots x_{p_{\ell-1},q_{\ell-1}})\delta_{p_{\ell},q_{\ell}} \\
&\overset{\textnormal{\ref{prop:epHQ}}\textnormal{\ref{eq:epHQ}}}{=}\delta_{t(q_{\ell-1}),s(q_{\ell})} \delta_{t(p_1),s(p_2)}\cdots\delta_{t(p_{\ell-2}),s(p_{\ell-1})}\delta_{t(q_1),s(q_2)}\cdots\delta_{t(q_{\ell-2}),s(q_{\ell-1})}\\
&\quad \quad \quad \quad \quad \cdot \delta_{p_1,q_1}\cdots \delta_{p_{\ell-1},q_{\ell-1}}\delta_{p_{\ell},q_{\ell}}\\
&=\delta_{t(p_{\ell-1}),s(p_{\ell})}\delta_{t(q_{\ell-1}),s(q_{\ell})} \delta_{t(p_1),s(p_2)}\cdots\delta_{t(p_{\ell-2}),s(p_{\ell-1})}\delta_{t(q_1),s(q_2)}\cdots\delta_{t(q_{\ell-2}),s(q_{\ell-1})}\\
&\quad \quad  \cdot \delta_{p_1,q_1}\cdots \delta_{p_{\ell-1},q_{\ell-1}}\delta_{p_{\ell},q_{\ell}}\\
&= \l(\delta_{t(p_1),s(p_2)}\cdots\delta_{t(p_{\ell-2}),s(p_{\ell-1})}\delta_{t(p_{\ell-1}),s(p_{\ell})}\r)\l(\delta_{t(q_1),s(q_2)}\cdots\delta_{t(q_{\ell-2}),s(q_{\ell-1})}\delta_{t(q_{\ell-1}),s(q_{\ell})}\r)\\
&\quad \quad \cdot \l(\delta_{p_1,q_1}\cdots \delta_{p_{\ell-1},q_{\ell-1}}\delta_{p_{\ell},q_{\ell}}\r)\\
&\overset{\textnormal{\ref{prop:epHQ}}\textnormal{\ref{eq:epHQ}}}{=}\ep_{\hay}(x_{p_1,q_1}\cdots x_{p_{\ell-1},q_{\ell-1}} x_{p_{\ell},q_{\ell}}),
\end{align*}
as desired. Therefore,
we have shown that $\pi$ is a map of weak bialgebras. 

\smallskip

It remains to show that $\pi$ is unique.
 Suppose that $\pi': \hay \to H$ is a weak bialgebra homomorphism such that $(\pi' \otimes \id_{\kk Q}) \lambda^{\hay} = \lambda^H$. Let $i \in Q_0$. Then
\[(\pi' \otimes \id_{\kk Q}) \lambda^{\hay}(e_i)  = (\pi' \otimes \id_{\kk Q})\left(\textstyle \sum_{j \in Q_0} x_{i,j} \otimes e_j\right)
= \textstyle \sum_{j \in Q_0} \pi'(x_{i,j}) \otimes e_j
\]
while
\[ \lambda^H(e_i) = \textstyle \sum_{j \in Q_0} y_{i,j} \otimes e_j.
\]
Since the $e_j$ are linearly independent in $\kk Q$, this implies that $\pi'(x_{i,j}) = y_{i,j}$ for each $i, j \in Q_0$.
By a similar argument, $\pi'(x_{p,q}) = y_{p,q}$ for any $p,q \in Q_1$. Hence, $\pi'$ and $\pi$ agree on a set of algebra generators for $\hay$, and since both $\pi'$ and $\pi$ are algebra homomorphisms, we have that $\pi' = \pi$.

\medskip

To show that $\hay$ satisfies the universal property of $\Oright(\kk Q)$, one only needs to make the following adjustments to the proof above: 
assume Hypothesis~\ref{hyp:coactions}\ref{hyp:coactions-R} in place of Hypothesis~\ref{hyp:coactions}\ref{hyp:coactions-L} (i.e., replace the left coaction $\lambda$ with the right coaction $\rho$); replace $\psi$ with an isomorphism $\phi: \kk Q_0 \overset{\sim}{\longrightarrow} H_s$ in $\AH$; and employ Proposition~\ref{prop:wba-idempotent-R} in place of Proposition~\ref{prop:wba-idempotent-L} in the argument that $\pi: \hay \to H$ is an algebra map. Then, the result for $\Oright(\kk Q)$ follows in a manner similar to that for $\Oleft(\kk Q)$ above.
\end{proof}

With Lemma~\ref{lem:O-in-faith}, the following is a consequence of the theorem above.

\begin{corollary}
The weak bialgebra $\hay$ coacts on $\kk Q$ inner-faithfully. \qed
\end{corollary}

We end this section with an example of our result above in the bialgebra case, thus obtaining a left/right/transposed UQSG as in Definitions~\ref{def:UQSG-intro} and~\ref{def:ManinUQSG-intro}. 

\begin{example} \label{ex: Manincase}
Suppose that $Q$ is a finite quiver with $|Q_0| =1$ and $|Q_1| = n$ for some $n \in \mathbb{N}$, that is, $Q$ is the $n$-loop quiver. Here, $\kk Q$ is isomorphic to the free algebra $\kk \langle t_1, \dots, t_n \rangle$. Now Theorem~\ref{thm:main-wba} implies that, as bialgebras,
$$O^{\textnormal{left}}(\kk \langle t_1, \dots, t_n \rangle) \cong O^{\textnormal{right}}(\kk \langle t_1, \dots, t_n \rangle) \cong O^{\textnormal{trans}}(\kk \langle t_1, \dots, t_n \rangle) \cong  \mathfrak{H}( Q_{n{\text{-loop}}}),$$
 where $\mathfrak{H}( Q_{n{\text{-loop}}})$ is defined in Example~\ref{ex:hay}. Indeed, $\dim_{\kk} (\mathfrak{H}(Q))_s = |Q_0| = 1$ by Proposition~\ref{prop:epHQ}\ref{prop:epHQs}, so all of the structures above are bialgebras by Proposition~\ref{prop:Hs-facts}\ref{itm:bialgebra}.
Moreover, one can check that $\mathfrak{H}(Q_{n{\text{-loop}}})$ is isomorphic to the free algebra  $\kk \langle x_{t_i, t_j} \mid 1 \leq i,j \leq n \rangle$.
\end{example}


\section{Universal quantum linear semigroupoids of quotients of path algebras} \label{sec:quotients}

Let $Q$ be a finite quiver and let $I$ be a graded ideal of $\kk Q$.
In this section, we study the UQSGds of the quotient algebra $\kk Q/I$, showing that if they exist, they are each a quotient of $\hay$ [\cref{prop:quotients}]. Moreover, we generalize a result of Manin by showing that a UQSGd of a quadratic quotient algebra is isomorphic to the opposite UQSGd of its quadratic dual [Theorem~\ref{thm:quad}]. We also provide several examples. To start, we need a few well-known facts.

\begin{definition}
Let $(H,m,u,\Delta,\ep)$ be a weak bialgebra. A \emph{biideal} of $H$ is a $\kk$-subspace $I \subseteq H$ which is both an ideal and a coideal, that is:  $hI \subseteq I$ and $Ih \subseteq I$ for any $h \in H$; 
 $\Delta(I) \subseteq I \otimes H + H \otimes I$; and
$\ep(I)=0$.
\end{definition}

\begin{lemma}
\label{lem:wbakernel}
The kernel of a weak bialgebra map is a biideal. 
\end{lemma}

\begin{proof}
Let $\alpha: H \to K$ be a weak bialgebra map. Since the kernel of an algebra map is an ideal and the kernel of a coalgebra map is a coideal, $\ker \alpha$ is a biideal.
\end{proof}

\begin{lemma}
\label{lem:biideal} 
Suppose that $H$ is a weak bialgebra and that $I$ is a biideal. Then $H/I$ can be given the structure of a weak bialgebra as follows, for all $h,k \in H$: $m_{H/ I}((h+I) \otimes (k+I))=hk+I$;  $1_{H/ I}=1_H + I$;  $\Delta_{H/I}(h+I)=(h_1 + I) \otimes (h_2+I)$; and $\ep_{H/I}(h+I):=\ep_H(h)$.
\end{lemma}

\begin{proof}
The structures given above make $H/I$ both an algebra and a coalgebra. A straightforward calculation verifies the compatibility conditions given in \cref{def:wba}. 
\end{proof}

\begin{proposition} 
\label{prop:quotients}
Let $Q$ be a finite quiver and let $I \subseteq \kk Q$ be a graded ideal which is generated in degree $2$ or greater. If $\Ocal^*(\kk Q/I)$ exists (where $*$ means `left', `right', or `trans'), we have $\Ocal^*(\kk Q/I) \cong \hay/\mathcal{I}$, for some biideal $\mathcal{I}$ of $\hay$. 
\end{proposition}
\begin{remark}
If $I$ has generators in degree 0 or 1, then we can choose a smaller quiver $Q'$ and an ideal $I'$ of $\kk Q'$ such that $\kk Q'/I' \cong \kk Q/I$ as algebras and $I'$ is generated in degree 2 or greater.
\end{remark}
\begin{proof}[Proof of Proposition~\ref{prop:quotients}]
We will prove this statement for $\Oleft(\kk Q/I)$; the other statements follow similarly. By \cref{lem:wbakernel}, it suffices to show that we have a weak bialgebra surjection $\pi: \hay \to \Oleft(\kk Q/I)$, in which case, $\Oleft(\kk Q/I) \cong \hay/\ker \pi$.

Let $\mathcal{O} := \Oleft(\kk Q/I)$. 
For $i\in Q_0$ and $p \in Q_1$, let $\bar{e_i},\:\bar{p}$ denote the images of $e_i,\: p$ in $\kk Q/I$ under the canonical quotient map $\kk Q \to \kk Q/I$ (regarding $p$ as an element of $\kk Q_1$).
Since $I$ is generated in degree 2 or greater, $(\kk Q/I)_0 \cong \kk Q_0$ as algebras, and $\dim_{\kk} (\kk Q/I)_1 = \dim_{\kk} \kk Q_1 = |Q_1|$.
Hence, $\{\bar{e_i} \}_{i \in Q_0}$ is a basis of $(\kk Q/I)_0$ and $\{\bar{p} \}_{p \in Q_1}$ is a basis of $(\kk Q/I)_1$.
We can write $1_{\kk Q/I}=\sum_{i\in Q_0} \bar{e_i}.$ Then we have a linear coaction
\begin{align*}
\lambda: \kk Q/I &\to \mathcal{O} \otimes \kk Q/I  \\
 \bar{e_j} &\mapsto \textstyle \sum_{i \in Q_0} y_{j,i} \otimes \bar{e_i} \\
 \bar{q} &\mapsto \textstyle \sum_{p \in Q_1} y_{q,p} \otimes \bar{p}  
\end{align*}
for some elements $y_{i,j}, y_{p,q} \in \mathcal{O}$. The result of \cref{lem:wba-structure}\ref{lem:wba-structure-a} holds for this coaction. Namely, the proof is the same, except we replace $\kk Q$ with $\kk Q/I$, elements of the form $e_i$ for $i\in Q_0$ with $\bar{e_i}$, and arrows $p \in Q_1$ (regarded as elements of $\kk Q_1$) with $\bar{p}$, making use of the fact that these elements of $\kk Q/I$ still satisfy the fundamental relations $\overline{e_i}\:\overline{e_j}=\delta_{i,j}\bar{e_i}$,  $\:\overline{e_{s(p)}}\:\overline{p}=\bar{p}=\overline{p}\:\overline{e_{t(p)}}$, for $i \in Q_0, p \in Q_1$. 
Therefore, the results of \cref{prop:wba-idempotent-L}\ref{itm:nonzero},\ref{itm:idem-L} also hold, since their proofs use the identities given in \cref{lem:wba-structure}\ref{lem:wba-structure-a}. 
Moreover, by the definition of a UQSGd, there exists a left $\mathcal{O}$-comodule algebra structure on $(\kk Q/I)_0$ such that $\mathcal{O}_t \cong (\kk Q/I)_0$ in ${}^\mathcal{O} \hspace{-.04in}\mathcal{A}$.
Therefore, if we replace $\kk Q_0$ with $(\kk Q/I)_0$ in the statement and proof of \cref{prop:wba-idempotent-L}\ref{itm:psi-L}, also replacing $e_i \in \kk Q_0$ with $\bar{e_i} \in (\kk Q/I)_0$, 
we obtain the same result. 

Now, imitating the proof of \cref{thm:main-wba}, we define a map $\pi$ defined on the algebra generators of $\hay$ and extended multiplicatively and linearly:
$$\pi: \hay \to \mathcal{O} \quad \text{defined by}  \quad
x_{i,j} \mapsto y_{i,j} \text{ for } i,j \in Q_0, \quad
x_{p,q} \mapsto y_{p,q} \text{ for } p,q \in Q_1.
$$
To show that $\pi$ is an algebra map, we can simply follow the proof for \cref{thm:main-wba}, since this proof only uses the results of \cref{lem:wba-structure}\ref{lem:wba-structure-a} and \cref{prop:wba-idempotent-L}. To show that $\pi$ is a coalgebra map, we again follow the proof for \cref{thm:main-wba}, 
replacing the paths $ p=p_1 \cdots p_{\ell-1}p_{\ell} $ and $ q=q_1 \cdots q_{\ell-1}q_{\ell}$
(for $p_i,q_i \in Q_1$) with their images under the canonical quotient map $\kk Q \to \kk Q/I$.
This proof only uses the results of \cref{lem:wba-structure}\ref{lem:wba-structure-a} and \cref{prop:wba-idempotent-L}, the weak bialgebra structure of $\hay$, and the fact that $\pi$ is multiplicative, so the result still holds. Therefore, $\pi$ is a weak bialgebra map.

Finally, we will show that $\pi$ is surjective. By \cref{lem:O-in-faith}, the coaction of $\mathcal{O}$ on $\kk Q/I$ is inner-faithful, and so $\mathcal{O}$ is generated as  a weak bialgebra by the $y_{i,j}$ and $y_{p,q}$ for $i,j \in Q_0$ and $p,q \in Q_1$. 
By the definition of $\pi$ and the fact that $\pi$ is a weak bialgebra map, we can see that $\pi$ is surjective. 
\end{proof}

Every connected graded $\kk$-algebra which is finitely generated in degree one is isomorphic to $\kk Q/I$ where $Q$ is a finite quiver with $|Q_0| = 1$. For these algebras, we obtain the following immediate corollary.

\begin{corollary}
If $Q$ is a finite quiver with $|Q_0| = 1$ and $|Q_1| = n$, then $\mathcal{O}^*(\kk Q/I)$ is a bialgebra quotient of the face algebra $\mathfrak{H}( Q_{n{\text{-loop}}})$ from Example~\ref{ex: Manincase}, where $*$ means `left', `right', or `trans'. \qed
\end{corollary}

The next example is a special case of \cref{prop:quotients}, which describes the UQSGds explicitly as a quotient of $\hay$ when $\kk Q/I$ is the polynomial ring $\kk[t_1, \dots, t_n]$.

\begin{example}\label{ex:polynomialManin}
Let $A = \kk[t_1, \dots, t_n]$. We can describe $A$ as a quotient of a path algebra $\kk Q / I$ where, $Q$ is a quiver with one vertex and $n$ arrows $t_1, \dots, t_n$, and $I = ([t_i, t_j])_{1 \leq i < j \leq n}$. 
Since $A$ is connected graded, as noted in \cref{rem:gen-Manin-oneside,rem:gen-Manin-twoside}, the UQSGds of $A$ are classical UQSGs (bialgebras).
\begin{enumerate}[(a),font=\upshape]
\item\label{itm:polynomial-Manin}
By \cite[Theorem~1]{AST}, we have that
\[O^{\textnormal{trans}} (\kk Q/ I) = O^{\textnormal{trans}} (\kk[t_1, \dots, t_n]) \cong O(\text{Mat}_n(\kk)).\] 
We will show that $O^{\textnormal{trans}} (\kk Q/I) \cong \hay/\mathcal{I}$, where $\hay$ is Hayashi's face algebra attached to an $n$-loop quiver from \cref{ex: Manincase} and $\mathcal{I}$ is the biideal of $\hay$ generated by the commutators $[x_{t_i, t_j}, x_{t_k, t_{\ell}}]$ for $1 \leq i,j,k,{\ell} \leq n$.

Let $O := O^{\textnormal{trans}} (A)$. Tracing through the two-sided version of \cref{prop:quotients}, we have coactions
\begin{align*}  
\lambda: A \to O\otimes A, \quad \lambda(t_i) = \textstyle{\sum_{j=1}^n} \; y_{t_i, t_j}\otimes t_j,\\
\rho: A \to A \otimes O, \quad \rho(t_i) = \textstyle{\sum_{j=1}^n} \; t_j\otimes y_{t_j, t_i}.
\end{align*}
Using the fact that $\lambda(t_i)\lambda(t_j) = \lambda(t_j)\lambda(t_i)$ and $\rho(t_i)\rho(t_j) = \rho(t_j)\rho(t_i)$ for all $1 \leq i,j \leq n$, one can show that all of the elements $y_{t_i, t_j}$ commute in $O$. Applying the coassociative and counital properties  of $\rho$ and $\lambda$, we can obtain the coalgebra structure on $O$, namely:
\[ \Delta(y_{t_i, t_j}) = \textstyle{\sum_{k=1}^n} \; y_{t_i, t_k} \otimes y_{t_k, t_i} \quad \text{ and } \quad \ep(y_{t_i, t_j}) = \delta_{i,j}.
\]
Thus, we can see that this presentation of $O$ agrees with the usual presentation of the bialgebra $O(\text{Mat}_n(\kk))$.

Now let $\pi: \mathfrak{H}(Q) \to O^{\textnormal{trans}} (A)$ be the surjective weak bialgebra map given in \cref{prop:quotients}, namely $\pi(x_{t_i,t_j})=y_{t_i, t_j}$. This is the canonical surjection of the free algebra on $n^2$ generators onto the polynomial algebra in $n^2$ variables and hence its kernel $\mathcal{I}$ is generated as a biideal by all commutators $[x_{t_i, t_j}, x_{t_k, t_{\ell}}]$.

\medskip

\item \label{itm:polynomial-leftright}
The left UQSGd $O^{\text{left}}(A)$ and right UQSGd $O^{\text{right}}(A)$ are the `half quantum groups' described in the introduction (see, e.g., \cite{CFR}). Explicitly, one can check that $O^{\text{left}}(A)$ is the quotient of $\hay$ by the biideal generated by 
\[
\quad \quad \quad \left\{[x_{t_i, t_j}, x_{t_k, t_j}]\right\}_{1 \leq i,j,k \leq n} \quad \text{ and }  \quad \left\{[x_{t_i, t_j}, x_{t_k, t_{\ell}}] - [x_{t_k, t_j}, x_{t_i, t_{\ell}}]\right\}_{1 \leq i,j,k,{\ell} \leq n \text{ with } j \neq l}.
\]
Similarly, $O^{\text{right}}(A)$ is the quotient of $\hay$ by the biideal generated by
\[
\quad \quad \quad \left\{[x_{t_i, t_j}, x_{t_i, t_k}]\right\}_{1 \leq i,j,k \leq n} \quad \text{ and }  \quad \left\{[x_{t_i, t_j}, x_{t_k, t_{\ell}}] - [x_{t_i, t_{\ell}}, x_{t_k, t_j}]\right\}_{1 \leq i,j,k,{\ell} \leq n \text{ with } i \neq k}.
\]
 \end{enumerate}
 Hence, when $A$ is a proper quotient of $\kk Q$, we need not have $O^{\textnormal{left}}(A) \cong O^{\textnormal{trans}} (A) \cong O^{\textnormal{right}}(A)$, in contrast with the path algebra case of \cref{thm:main-wba}.
\end{example}

 Next we turn our attention to UQSGds of quadratic quotient algebras $\kk Q/I$. Consider the following terminology.

\begin{definition}[{\cite[Section~2]{GMV}}, {\cite[Section~1]{MV}}, {\cite{Gawell}}]
\label{def:quaddual} Let $Q$ be a finite quiver and suppose $I$ is a graded ideal of the path algebra $\kk Q$.
\begin{enumerate} 
    \item The {\it opposite quiver} $Q^{\text{op}}$ of $Q$ is defined to be the quiver formed by $(Q^{\text{op}})_0 = Q_0$ and $(Q^{\text{op}})_1 = Q_1^*$, where $Q_1^*$ is the arrow set consisting of reversed arrows of $Q_1$. For $p \in Q_1$, its reverse in $Q_1^*$ is denoted by $p^*$.
    If $a = p_1 \dots p_{\ell}$ is a path of length $\ell$ in $Q$, then we let $a^* = p_{\ell}^* \dots p_1^* \in Q^{\text{op}}$.
    If $f = \sum_i \alpha_i a_i$ is an element of $\kk Q$, the element $f^* \in \kk Q^{\text{op}}$ is defined to be $\sum_i \alpha_i a_i^*$.
    
    \smallskip 
    
    \item We identify $\kk Q^{\op}_\ell$ with $(\kk Q_\ell)^*$ so that if $\{a_1, \dots, a_d\}$ is the basis of $\kk Q_\ell$ consisting of paths of length $\ell$, then $\{a_1^*, \dots, a_d^*\}$ is the dual basis.
         
     \smallskip 
    
    \item We call the quotient algebra $\kk Q/I$ {\it quadratic} if  $I$ is generated by elements of $\kk Q_2$.
    
    \smallskip
    
    \item The {\it quadratic dual}  of the quadratic algebra $\kk Q/I$ is defined to be
    $$(\kk Q / I)^! = \kk Q^{\text{op}}/I_{\text{op}}^{\perp},$$
    where $I_{\text{op}}^\perp$ is the ideal of $\kk Q^{\text{op}}$ generated by the orthogonal complement of the set $I_{\text{op}}:=\{f^* \in \kk Q^{\text{op}} \mid f \in I \cap \kk Q_2\}$ in $\kk Q^{\text{op}}_2$.
    \end{enumerate}
\end{definition}

\begin{remark}
As is our convention of Notation~\ref{not:quiver}, we still read paths from left-to-right in $Q^{\text{op}}$. Hence, in $\kk Q^{\text{op}}$ we have
\[ q^* p^* = (pq)^*
\]
for $p,q \in Q$ (which is nonzero when $s(p^*) = t(p) = s(q) = t(q^*)$). Note that identifying $p \in \kk Q$ with $p^* \in \kk Q^\text{op}$ yields an anti-isomorphism of algebras and so $\kk Q^{\text{op}} \cong (\kk Q)^{\text{op}}$.

For the face algebras $\hay$ and $\mathfrak{H}(Q^{\text{op}})$ attached to $Q$ and $Q^{\text{op}}$, respectively, the map which sends $x_{a,b} \in \hay$ to $x_{a^*, b^*} \in \mathfrak{H}(Q^{\text{op}})$ is an anti-isomorphism of algebras and an isomorphism of coalgebras. As weak bialgebras, $\mathfrak{H}(Q^{\text{op}}) \cong \hay^{\text{op}}$. 
\end{remark}

The following theorem is a non-connected generalization of \cite[Theorem~5.10]{Manin}.

\begin{theorem} \label{thm:quad}
Let $Q$ be a finite quiver and suppose $I$ is an ideal such that $\kk Q/I$ is quadratic. Then, we have that
\begin{enumerate} [(a),font=\upshape] 
    \item \label{itm:left-right} $\Oleft(\kk Q/I) \cong  \Oright((\kk Q/I)^!)^{\op}$,
    \smallskip
    \item \label{itm:right-left} $\Oright(\kk Q/I) \cong \Oleft((\kk Q/I)^!)^{\op}$,
    \smallskip
    \item \label{itm:left-right-same} $\Oleft(\kk Q/I) \cong \Oright(\kk Q/I)^{\cop}$,
    \smallskip
    \item \label{itm:Manin-Manin} $\Otrans(\kk Q/I) \cong \Otrans((\kk Q/I)^!)^{\op}$,
\end{enumerate}
as weak bialgebras. 
\end{theorem}

\begin{proof}
We will only provide the proofs of parts \ref{itm:left-right} and \ref{itm:left-right-same}, as other parts will hold by similar arguments.  To start, suppose that $Q_1 = \{p_1, \dots, p_n\}.$ Then, 
\[
I = \left\langle r_\alpha:= \textstyle \sum_{i,j = 1 \text{ with } t(p_i) = s(p_j)}^n c_{i,j}^{[\alpha]}\; p_i \;p_j \right\rangle_{\alpha = 1,  \dots, m}\; \subseteq \kk Q_2
\]
for some scalars $c_{i,j}^{[\alpha]}$. Moreover, we have
\[I_{\text{op}}^\perp = \left\langle  r_{\beta}^* := \textstyle \sum_{k,\ell = 1 \text{ with } t(p_\ell^*) = s(p_k^*)}^n d_{k,\ell}^{[\beta]} \; p_{\ell}^* p_k^* \right\rangle_{\beta = 1,  \dots, |Q_2|-m}\; \subseteq \kk Q_2^{\text{op}}
\]
for some scalars $d_{k,\ell}^{[\beta]}$. Here,
$\textstyle \sum_{i,j = 1 \text{ with } t(p_i) = s(p_j)}^n d_{i,j}^{[\beta]} \; c_{i,j}^{[\alpha]} = 0$
for each pair $\alpha, \beta$. 

\medskip 

\ref{itm:left-right} By Proposition~\ref{prop:quotients}, we have that
$\Oleft(\kk Q/I) = \hay/ \mathcal{I}$ for some biideal $\mathcal{I}$ of $\hay$, with the coalgebra structure induced by $\hay :\Delta(x_{p_i, p_k}) = \sum_{w=1}^n x_{p_i, p_w} \otimes x_{p_w, p_k}$ and $\varepsilon(x_{p_i, p_k}) = \delta_{i,k}$. We assert that 
$$\mathcal{I} = \left\langle  \sum_{\substack{i,j,k,\ell = 1\\ t(p_i) = s(p_j),\; t(p_k) = s(p_\ell)}}^n c_{i,j}^{[\alpha]}\; d_{k,\ell}^{[\beta]} \;x_{p_i,p_k}\; x_{p_j,p_\ell}\right\rangle_{\substack{\alpha = 1,\dots, m \\ \beta = 1, \dots, |Q_2|-m}}.$$

\smallskip

\noindent Namely, there exists a basis of $\kk Q_2$ consisting of elements $\{r_{\alpha}\}_{\alpha=1,\dots,m}$ and $\{s_{\gamma}\}_{\gamma=1, \dots, |Q_2|-m}$ so that the evaluation $\langle r_{\beta}^*, s_\gamma \rangle = \delta_{\beta,\gamma}$
for each $\beta,\gamma = 1,\dots, |Q_2|-m$. Moreover, for each $k,\ell$, we can write
\begin{equation} \label{eq:quad}
   \textstyle p_k p_\ell = \sum_\gamma d_{k,\ell}^{[\gamma]} \; s_\gamma + \sum_\alpha e_{k,\ell}^{[\alpha]} \; r_\alpha
\end{equation}
for some scalars $e_{k,\ell}^{[\alpha]}$. (This can be checked by evaluation with $r_{\beta}^*$.)
Now, the left coaction of $\hay$ on $\kk Q$, given by $p_i \mapsto \sum_k x_{p_i,p_k} \otimes p_k$ from Example~\ref{ex:haycoaction}, preserves the relation $r_\alpha$ if and only if the following expression lies in $\mathcal{O} \otimes I$:
\[
\begin{array}{l}
\medskip
\textstyle \sum_{i,j} c_{i,j}^{[\alpha]} (\sum_k x_{p_i,p_k} \otimes p_k)(\sum_\ell x_{p_j,p_\ell} \otimes p_\ell)\\
\medskip
= \textstyle \sum_{\substack{i,j,k,l\\ t(p_i) = s(p_j)\\ t(p_k) = s(p_\ell)}} c_{i,j}^{[\alpha]}\; x_{p_i,p_k} \; x_{p_j,p_\ell} \otimes p_k p_\ell\\
\overset{\eqref{eq:quad}}{=} \textstyle \sum_{\substack{\gamma, i,j,k,l\\ t(p_i) = s(p_j)\\ t(p_k) = s(p_\ell)}} c_{i,j}^{[\alpha]}\; d_{k,\ell}^{[\gamma]}\; x_{p_i,p_k} \; x_{p_j,p_\ell} \otimes s_\gamma
\; + \; \sum_{\substack{\alpha',i,j,k,l\\ t(p_i) = s(p_j)\\ t(p_k) = s(p_\ell)}} c_{i,j}^{[\alpha]}\; e_{k,\ell}^{[\alpha']}\; x_{p_i,p_k} \; x_{p_j,p_\ell} \otimes r_{\alpha'} .
\end{array}
\]
Since $\{s_{\gamma}\}_{\gamma} \cup \{r_{\alpha'} \}_{\alpha '}$ is a basis of $\kk Q_2$, we must have that 
\[\bigg\{ \sum_{\substack{i,j,k,\ell=1\\ t(p_i) = s(p_j)\\ t(p_k) = s(p_\ell)}}^n c_{i,j}^{[\alpha]}\; d_{k,\ell}^{[\gamma]}\; x_{p_i,p_k} \; x_{p_j,p_\ell} = 0\bigg\}_{\substack{\alpha = 1,\dots, m\\ \gamma = 1, \dots, |Q_2|-m}}\]
are the generators of the relation space $\mathcal{I}$ for $\Oleft(\kk Q/I)$ as in Proposition~\ref{prop:quotients}. 
 
On the other hand, by Example~\ref{ex:haycoaction} we  have a right coaction of $\hayop$ on $\kk Q^\op$ given by $p_i^* \mapsto \sum_k p_k^* \otimes x_{p_k^*, p_i^*}$. By a similar argument as above, this coaction preserves the relations of $I_{\op}^\perp$ if and only if
\[\bigg\{ \sum_{\substack{i,j,k,\ell = 1\\ t(p_j^*) = s(p_i^*)\\ t(p_{\ell}^*) = s(p_k^*)}}^n d_{k,\ell}^{[\beta]} c_{i,j}^{[\eta]}\; \; x_{p_{j}^*,p_{\ell}^*} \; x_{p_i^*,p_k^*} = 0\bigg\}_{\substack{\eta = 1,\dots, m\\ \beta = 1, \dots, |Q_2|-m}}
\]
are the generators of the relation space of $\Oright( (\kk Q/I)^!)$ as a quotient of $\hayop$. Hence,
\[\Oright((\kk Q/I)^!) = \hayop/ \left\langle \sum_{\substack{i,j,k,\ell = 1\\ t(p_j^*) = s(p_i^*), \; t(p_{\ell}^*) = s(p_k^*)}}^n c_{i,j}^{[\alpha]} d_{k,\ell}^{[\beta]} \; \; x_{p_{j}^*,p_{\ell}^*} \; x_{p_i^*,p_k^*} \right\rangle_{\substack{\alpha = 1,\dots, m\\ \beta = 1, \dots, |Q_2|-m}},
\]

\smallskip

\noindent with $\Delta(x_{p_i^*, p_k^*}) = \sum_{w=1}^n x_{p_i^*, p_w^*} \otimes x_{p_w^*, p_k^*}$  and $\varepsilon(x_{p_i^*, p_k^*}) = \delta_{i,k}$.

Now, the desired isomorphism from $\Oleft(\kk Q/I)$ to  $\Oright((\kk Q/I)^!)^{\op}$ is obtained by sending $x_{p_i,p_k}$ to $x_{p_i^*,p_k^*}$.

\medskip

\ref{itm:left-right-same} By Proposition~\ref{prop:quotients}, we have that
$\Oright(\kk Q/I) = \hay/ \mathcal{I}'$ for some biideal $\mathcal{I}'$ of $\hay$, with the coalgebra structure induced by $\hay :\Delta(x_{p_k, p_i}) = \sum_{w=1}^n x_{p_k, p_w} \otimes x_{p_w, p_i}$ and $\varepsilon(x_{p_k, p_i}) = \delta_{i,k}$. Now, the right coaction of $\hay$ on $\kk Q$, given by $p_i \mapsto \sum_k  p_k \otimes x_{p_k,p_i}$ from Example~\ref{ex:haycoaction}, preserves each relation $r_\alpha$ of $I$ if and only if 
$$\mathcal{I}' = \left\langle  \sum_{\substack{i,j,k,\ell = 1\\ t(p_i) = s(p_j),\; t(p_k) = s(p_\ell)}}^n c_{i,j}^{[\alpha]}\; d_{k,\ell}^{[\beta]} \;x_{p_k,p_i}\; x_{p_\ell,p_j}\right\rangle_{\substack{\alpha = 1,\dots, m\\ \beta = 1, \dots, |Q_2|-m}}.$$

\smallskip

\noindent Now, considering the presentation of $\Oleft(\kk Q/I)$ from part~\ref{itm:left-right}, the desired isomorphism  from $\Oleft(\kk Q/I)$ to  $\Oright(\kk Q/I)^{\cop}$ is obtained by sending $x_{p_i,p_k}$ to $x_{p_k,p_i}$.
\end{proof}

\begin{example}
Let $A = \kk Q$. Then $A^! = \kk Q^{\op}/ \langle \kk (Q^{\op})_2 \rangle$ where $\langle \kk (Q^{\text{op}})_2 \rangle$ is the ideal of $\kk Q^{\text{op}}$ generated by the space $\kk (Q^{\text{op}})_2$. By the above theorem, we have
\[\Oleft(A)\cong \Oright(A^!)^{\op}
\]
as weak bialgebras.  Since, by Theorem~\ref{thm:main-wba}, $\Oleft(\kk Q) \cong \hay$, we have that $\Oright(A^!)^{\op} \cong \hay$. Further, $\hay^{\op} \cong \hayop$, and so we conclude that 
\[ \Oright(A^!) \cong \hayop.
\]
Similarly, we have that $\Oleft(A^!) \cong \Otrans(A^!) \cong \hayop$ as weak bialgebras.
\end{example}

We end with a family of concrete examples of UQSGds for quadratic quotient path algebras-- namely, those for preprojective algebras.

\begin{example}
Let $Q$ be the extended type $A$ Dynkin quiver with $|Q_0|\geq 3$, and consider its {\it double} $\bar{Q}$ formed by adding $p^*$ for each $p \in Q_1$. For example, when $|Q_0|=3$, 
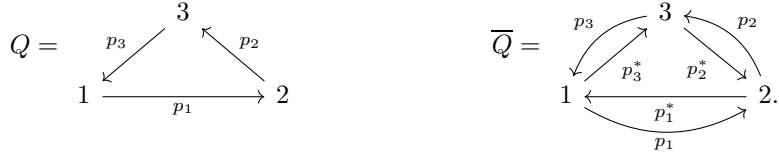
\[Q=\begin{tikzcd}
{} & 3 \arrow[ld,swap,"p_3"] & {}\\
1 \arrow[rr, swap,"p_1"] & {} & 2 \arrow[lu, swap,"p_2"]
\end{tikzcd}
\hspace{1in} \bar{Q}=\begin{tikzcd}
{} & 3 \arrow[rd,  swap,"p_2^*"]\arrow[ld,swap, bend right, "p_3"] & {}\\
1 \arrow[ru,swap,  "p_3^*"]\arrow[rr, bend right, swap,"p_1"] & {} & 2. \arrow[ll,  "p_1^*"]\arrow[lu, bend right, swap,"p_2"]
\end{tikzcd}\]
The {\it preprojective algebra} on $Q$ is defined to be the $\kk$-algebra, 
\[\Pi_Q = \kk \bar{Q}/ \textstyle (\sum_{i \in Q_0} p_ip_i^* - \sum_{i \in Q_0} p_{i-1}^*p_{i-1}).\]
(Here, we index the vertices $i$ by elements of $\mathbb{Z}/|Q_0|\mathbb{Z}$.)
By \cite[Section~3]{Weis}, we have that
\[\Pi_Q \cong \kk \bar{Q}/(p_ip_i^* - p_{i-1}^*p_{i-1})_{i \in Q_0}\]
as $\kk$-algebras.
Therefore, any path in $Q$ can be rewritten (in $\Pi_Q$) so that all of the nonstar arrows occur, followed by all of the star arrows.
We omit the details here, but we have that, as weak bialgebras,
$$\mathcal{O^{\textnormal{left}}}(\Pi_Q) \cong \mathfrak{H}(\bar{Q})/\mathcal{I}, 
\quad \quad \; 
\mathcal{O^{\textnormal{right}}}(\Pi_Q) \cong \mathfrak{H}(\bar{Q})/\mathcal{J},
\quad \quad \;  
\mathcal{O^{\textnormal{trans}}}(\Pi_Q) \cong \mathfrak{H}(\bar{Q})/(\mathcal{I} + \mathcal{J}),$$
for 
\[
\mathcal{I}=
\left\langle
\begin{array}{c}
x_{p_k,p_i}x_{p_k^*,p_{i+1}}-x_{p_{k-1}^*,p_i}x_{p_{k-1},p_{i+1}},\\ 
(x_{p_k,p_i}x_{p_k^*,p_i^*}+ x_{p_k,p_{i-1}^*}x_{p_k^*,p_{i-1}})-( x_{p_{k-1}^*,p_i}x_{p_{k-1},p_i^*}+x_{p_{k-1}^*,p_{i-1}^*}x_{p_{k-1},p_{i-1}}), \\
x_{p_k,p_i^*} x_{p_k^*,p_{i-1}^*}-x_{p_{k-1}^*,p_i^*}x_{p_{k-1},p_{i-1}^*}
\end{array}
\right\rangle_{i,k \in Q_0},
\]
\medskip
\[
   \mathcal{J}= 
     \left\langle 
     \begin{array}{c}
     x_{p_i,p_k}x_{p_{i+1},p_{k}^*}-x_{p_i,p_{k-1}^*}x_{p_{i+1},p_{{k-1}}},\\
   (x_{p_i,p_k}x_{p_i^*,p_k^*}+ x_{p_{i-1}^*,p_k}x_{p_{i-1},p_k^*})-(x_{p_i,p_{k-1}^*}x_{p_i^*,p_{k-1}}+ x_{p_{i-1}^*,p_{k-1}^*}x_{p_{i-1},p_{k-1}}),\\
   x_{p_i^*,p_k} x_{p_{i-1}^*,p_k^*}- x_{p_i^*,p_{k-1}^*} x_{p_{i-1}^*,p_{k-1}}.
    \end{array}
     \right\rangle _{i,k \in Q_0}.
\]
\end{example}

\medskip

\begin{remark} \label{rem:N-homog}
The universal quantum semigroupoids analyzed in Theorem~\ref{thm:quad} generalize Manin's universal semigroups for quadratic algebras \cite{Manin}, and likewise, universal quantum semigroupoids  exist and can be constructed explicitly for $N$-homogeneous algebras. This, in turn, generalizes Chirvasitu--Walton--Wang's universal quantum semigroups for such algebras \cite{CWW} to the non-connected setting. Prevalent non-connected $N$-homogeneous algebras include quiver (super)potential algebras, which are used frequently in cluster theory \cite{DWZ}, in Donaldson--Thomas theory \cite{MosRei}, and in other fields; we expect that the universal symmetries of these algebras will have similar applications.
\end{remark}

\section{For further investigation: universal quantum linear groupoids} \label{sec:UQGd}

In this section, we consider weak \emph{Hopf} algebras that coact universally (and linearly) on an algebra $A$ in Hypothesis~\ref{hyp:A}, and propose directions for future research. First, let us recall the notion of a universal coacting Hopf algebra, prompted by \cite[Chapter~7]{Manin}.

\begin{definition}[UQG]
Take $A$ as in Hypothesis~\ref{hyp:A} and further assume that $A$ is connected. Then a Hopf algebra  is said to be a {\it left (resp., right, transposed) universal quantum linear group (UQG) of $A$} if it satisfies the conditions of Definition~\ref{def:UQSG-intro}\ref{itm:leftUQSG-intro} (or, Definition~\ref{def:UQSG-intro}\ref{itm:rightUQSG-intro}, Definition~\ref{def:ManinUQSG-intro}\ref{itm:ManinUQSG-intro}) by replacing `bialgebra' with `Hopf algebra'.
\end{definition}

A general way of constructing a UQG from a UQSG is by taking the {\it Hopf envelope} as discussed briefly in \cite[Section~7.5]{Manin}. Other explicit constructions involve the {\it quantum determinant} (also known as the {\it homological determinant}), which is a (typically) central group-like element $\mathsf{D}$ of a UQSG that depends on the UQSG coaction on $A$. Here, one takes a UQSG, say $O^{\textnormal{trans}}(A)$, and forms two Hopf algebras depending on whether the quantum determinant is trivial (i.e., equal to the unit) or is arbitrary:
$$ O_{\text{SL}}^{\textnormal{trans}}(A):= O^{\textnormal{trans}}(A)/ (\mathsf{D} - 1), \quad \quad  O_{\text{GL}}^{\text{trans}}(A):= O^{\textnormal{trans}}(A)[\mathsf{D}^{-1}].$$
We refer to these universal Hopf algebras as UQGs {\it of SL-type} and of {\it GL-type}, respectively. Appearances of such Hopf algebras in the literature include those in \cite{DVL, Bichon, BichonDV} for SL-type, \cite{AST, Takeuchi, Mrozinski} for GL-type, and \cite{WaltonWang, CWW} for both  types; see also references therein.

\smallskip

It is therefore natural to ask if this can be generalized to the framework of universal coacting weak Hopf algebras. We recall the definition of a weak Hopf algebra below.

\begin{definition} \label{def:wha}
A \textit{weak Hopf algebra} is a sextuple $(H,m,u,\Delta,\varepsilon, S)$, where $(H,m,u,\Delta,\varepsilon)$ is a weak bialgebra and $S: H \to H$ is a $\kk$-linear map called the \textit{antipode} that satisfies the following properties for all $h\in H$:
 $$S(h_1)h_2=\varepsilon_s(h), \quad \quad
h_1S(h_2)=\varepsilon_t(h), \quad \quad
S(h_1)h_2S(h_3)=S(h).$$
\end{definition}

Note that if $H$ is a weak Hopf algebra, the following are equivalent:
 $H$ is a Hopf algebra;
 $\Delta(1)=1\otimes 1$;
 $\ep(xy)=\ep(x)\ep(y)$ for all $x,y\in H$;
$S(x_1)x_2=\ep(x)1$ for all $x \in H$; and
 $x_1S(x_2)=\ep(x)1$ for all $x \in H$ \cite[page~5]{BNS}. 

\smallskip

Now we define a universal weak Hopf algebra, similar to the manner that a UQG was defined above, for $A$ not necessarily connected.

\begin{definition}[UQGd]
Take $A$ as in Hypothesis~\ref{hyp:A}. Then a weak Hopf algebra is said to be a {\it left (resp., right, transposed) universal quantum linear groupoid of $A$} if it satisfies the conditions of part~\ref{itm:leftUQSGd-intro} (resp., \ref{itm:rightUQSGd-intro}, \ref{itm:ManinUQSGd-intro}) of Definition~\ref{def:UQSGd-intro} by replacing `weak bialgebra' with `weak Hopf algebra'.
\end{definition}

This prompts the following series of questions.

\begin{question} \label{ques:UQGds} Take $A$ as in Hypothesis~\ref{hyp:A}.
In general:
\begin{enumerate}[(1),font=\upshape] 
\item When are the UQSGds $\Oleft(A)$, $\Oright(A)$, $\Otrans(A)$ weak Hopf algebras?
\smallskip
\item \label{ques:wha-env} What is a `weak Hopf envelope' (of a UQSGd)?
\end{enumerate}
Pertaining to the SL-type and GL-type constructions:
\begin{enumerate}[(1),font=\upshape] 
\setcounter{enumi}{2}
\item \label{ques:wha-det} What is the quantum determinant $\mathsf{D}$ of the coaction of a UQSGd $\Oleft(A)$ (or, $\Oright(A)$, $\Otrans(A)$) on $A$? 
\smallskip
\item \label{ques:wha-GL} Is $\mathsf{D}$ invertible, and if so, are $\Oleft(A)[\mathsf{D}^{-1}]$,  $\Oright(A)[\mathsf{D}^{-1}]$, $\Otrans(A)[\mathsf{D}^{-1}]$ weak Hopf algebras that coact on $A$ (universally) with arbitrary quantum determinant?
\smallskip
\item Are $\Oleft(A)/(\mathsf{D}-1)$,  $\Oright(A)/(\mathsf{D}-1)$, $\Otrans(A)/(\mathsf{D}-1)$ weak Hopf algebras that coact on $A$ (universally) with trivial quantum determinant?
\end{enumerate}
\end{question}


\smallskip

Finally, as discussed in the introduction: 
\begin{center} 
{\it Ideally, a universal (weak) bi/Hopf algebra should behave ring-theoretically\\ and homologically like the algebra that it coacts on.}
\end{center}
This holds for transposed coactions on many connected graded algebras in Hypothesis~\ref{hyp:A}; see, e.g., \cite{AST, WaltonWang}. Likewise, the best candidates  we have for the philosophy to hold for coactions on algebras in Hypothesis~\ref{hyp:A} are the transposed UQSGds [Definition~\ref{def:manin-uqsgd}] and the transposed UQGds (defined above). Naturally, we inquire:

\begin{question} \label{ques:ring-hom}
Take an algebra $A$ in Hypothesis~\ref{hyp:A}. In general, does the ring-theoretic and homological behavior of the transposed UQSGd and transposed UQGds of $A$ reflect that of $A$? More specifically, if $A$ has one of the following properties,
\begin{enumerate}
    \item finite Gelfand-Kirillov dimension/nice Hilbert series,
    \item Noetherian/coherent,
    \item domain/prime/semiprime,
    \item finite global dimension/finite injective dimension,
    \item skew (or twisted) Calabi-Yau,
\end{enumerate}
do the transposed UQSGd and transposed UQGds of $A$ satisfy the same property as well?
\end{question}

\noindent Pertinent articles include the work of Gaddis, Reyes, Rogalski, and Zhang on (non-connected) graded skew Calabi-Yau algebras \cite{RRZ,RR2018,RR2019,GR}. 

\begin{remark} \label{rem:fabio}
After the first version of this article appeared, Question~\ref{ques:ring-hom} was addressed in work of Calderón and Walton \cite{CalWal} in the case that $A$ is the path algebra $\kk Q$ of a finite quiver. It was shown that many properties hold for $\kk Q$ if and only if they hold for its UQSGd $\hay$, including: having finite dimension, having finite Gelfand--Kirillov dimension, being noetherian, being semiprime, and being Koszul. Graph-theoretic techniques were used to achieve these results.
\end{remark}

\bibliographystyle{alpha}
\bibliography{biblio}

\begin{thebibliography}{WWW22}

\bibitem[ASS06]{ASS2006}
Ibrahim Assem, Daniel Simson, and Andrzej Skowro{\'n}ski.
\newblock {\em Elements of the representation theory of associative algebras.
  {V}ol. 1}, volume~65 of {\em London Mathematical Society Student Texts}.
\newblock Cambridge University Press, Cambridge, 2006.
\newblock Techniques of representation theory.

\bibitem[AST91]{AST}
Michael Artin, William Schelter, and John Tate.
\newblock Quantum deformations of {${\rm GL}_n$}.
\newblock {\em Comm. Pure Appl. Math.}, 44(8-9):879--895, 1991.

\bibitem[BCG22]{BCG}
Suvrajit Bhattacharjee, Alexandru Chirvasitu, and Debashish Goswami.
\newblock Quantum {G}alois groups of subfactors.
\newblock {\em Internat. J. Math.}, 33(2):Paper No. 2250013, 30, 2022.

\bibitem[BCJ11]{BCJ}
G.~B\"{o}hm, S.~Caenepeel, and K.~Janssen.
\newblock Weak bialgebras and monoidal categories.
\newblock {\em Comm. Algebra}, 39(12):4584--4607, 2011.

\bibitem[BDV13]{BichonDV}
Julien Bichon and Michel Dubois-Violette.
\newblock The quantum group of a preregular multilinear form.
\newblock {\em Lett. Math. Phys.}, 103(4):455--468, 2013.

\bibitem[Bic03]{Bichon}
Julien Bichon.
\newblock The representation category of the quantum group of a non-degenerate
  bilinear form.
\newblock {\em Comm. Algebra}, 31(10):4831--4851, 2003.

\bibitem[BNS99]{BNS}
Gabriella B\"ohm, Florian Nill, and Korn\'el Szlach\'anyi.
\newblock Weak {H}opf algebras. {I}. {I}ntegral theory and {$C^*$}-structure.
\newblock {\em J. Algebra}, 221(2):385--438, 1999.

\bibitem[BSW10]{BSW}
Raf Bocklandt, Travis Schedler, and Michael Wemyss.
\newblock Superpotentials and higher order derivations.
\newblock {\em J. Pure Appl. Algebra}, 214(9):1501--1522, 2010.

\bibitem[CFR09]{CFR}
A.~Chervov, G.~Falqui, and V.~Rubtsov.
\newblock Algebraic properties of {M}anin matrices. {I}.
\newblock {\em Adv. in Appl. Math.}, 43(3):239--315, 2009.

\bibitem[CW]{CalWal}
Fabio Calderón and Chelsea Walton.
\newblock Algebraic properties of face algebras.
\newblock ArXiv preprint math/2105.12042, 2021. To appear in {\it J. Alg.
  Appl.}

\bibitem[CWW19]{CWW}
Alexandru Chirvasitu, Chelsea Walton, and Xingting Wang.
\newblock On quantum groups associated to a pair of preregular forms.
\newblock {\em J. Noncommut. Geom.}, 13(1):115--159, 2019.

\bibitem[DVL90]{DVL}
Michel Dubois-Violette and Guy Launer.
\newblock The quantum group of a nondegenerate bilinear form.
\newblock {\em Phys. Lett. B}, 245(2):175--177, 1990.

\bibitem[DWZ08]{DWZ}
Harm Derksen, Jerzy Weyman, and Andrei Zelevinsky.
\newblock Quivers with potentials and their representations. {I}. {M}utations.
\newblock {\em Selecta Math. (N.S.)}, 14(1):59--119, 2008.

\bibitem[EGNO15]{EGNO}
Pavel Etingof, Shlomo Gelaki, Dmitri Nikshych, and Victor Ostrik.
\newblock {\em Tensor categories}, volume 205 of {\em Mathematical Surveys and
  Monographs}.
\newblock American Mathematical Society, Providence, RI, 2015.

\bibitem[EN01]{EN}
Pavel Etingof and Dmitri Nikshych.
\newblock Dynamical quantum groups at roots of 1.
\newblock {\em Duke Math. J.}, 108(1):135--168, 2001.

\bibitem[ENO05]{ENO}
Pavel Etingof, Dmitri Nikshych, and Viktor Ostrik.
\newblock On fusion categories.
\newblock {\em Ann. of Math. (2)}, 162(2):581--642, 2005.

\bibitem[Gaw14]{Gawell}
Elin Gawell.
\newblock {\em Centra of Quiver Algebras}.
\newblock Licentiate thesis, Stockholm University, 2014.

\bibitem[GMV98]{GMV}
Edward~L. Green and Roberto Mart\'{\i}nez-Villa.
\newblock Koszul and {Y}oneda algebras. {II}.
\newblock In {\em Algebras and modules, {II} ({G}eiranger, 1996)}, volume~24 of
  {\em CMS Conf. Proc.}, pages 227--244. Amer. Math. Soc., Providence, RI,
  1998.

\bibitem[GP79]{G-P}
I.~M. Gel'fand and V.~A. Ponomarev.
\newblock Model algebras and representations of graphs.
\newblock {\em Funktsional. Anal. i Prilozhen.}, 13(3):1--12, 1979.

\bibitem[GR21]{GR}
Jason Gaddis and Daniel Rogalski.
\newblock Quivers supporting twisted {C}alabi-{Y}au algebras.
\newblock {\em J. Pure Appl. Algebra}, 225(9):Paper No. 106645, 33, 2021.

\bibitem[Hay93]{Hayashi93}
Takahiro Hayashi.
\newblock Quantum group symmetry of partition functions of {IRF} models and its
  application to {J}ones' index theory.
\newblock {\em Comm. Math. Phys.}, 157(2):331--345, 1993.

\bibitem[Hay96]{Hayashi96}
Takahiro Hayashi.
\newblock Compact quantum groups of face type.
\newblock {\em Publ. Res. Inst. Math. Sci.}, 32(2):351--369, 1996.

\bibitem[Hay99]{Hayashi99b}
Takahiro Hayashi.
\newblock Face algebras and unitarity of {${\rm SU}(N)_L$}-{TQFT}.
\newblock {\em Comm. Math. Phys.}, 203(1):211--247, 1999.

\bibitem[LT07]{LauveTaft}
Aaron Lauve and Earl~J. Taft.
\newblock A class of left quantum groups modeled after {${\rm SL}_q(n)$}.
\newblock {\em J. Pure Appl. Algebra}, 208(3):797--803, 2007.

\bibitem[Man88]{Manin}
Yu.~I. Manin.
\newblock {\em Quantum groups and noncommutative geometry}.
\newblock Universit\'{e} de Montr\'{e}al, Centre de Recherches
  Math\'{e}matiques, Montreal, QC, 1988.

\bibitem[MR10]{MosRei}
Sergey Mozgovoy and Markus Reineke.
\newblock On the noncommutative {D}onaldson-{T}homas invariants arising from
  brane tilings.
\newblock {\em Adv. Math.}, 223(5):1521--1544, 2010.

\bibitem[Mro14]{Mrozinski}
Colin Mrozinski.
\newblock Quantum groups of {$\rm GL(2)$} representation type.
\newblock {\em J. Noncommut. Geom.}, 8(1):107--140, 2014.

\bibitem[MV07]{MV}
Roberto Mart\'{\i}nez-Villa.
\newblock Introduction to {K}oszul algebras.
\newblock {\em Rev. Un. Mat. Argentina}, 48(2):67--95 (2008), 2007.

\bibitem[Nik02]{Nik02}
Dmitri Nikshych.
\newblock On the structure of weak {H}opf algebras.
\newblock {\em Adv. Math.}, 170(2):257--286, 2002.

\bibitem[{Nil}]{Nill}
Florian {Nill}.
\newblock {Axioms for Weak Bialgebras}.
\newblock ArXiv preprint math/9805104, 1998.

\bibitem[NV00]{NV2000}
Dmitri Nikshych and Leonid Vainerman.
\newblock A characterization of depth 2 subfactors of {${\rm II}_1$} factors.
\newblock {\em J. Funct. Anal.}, 171(2):278--307, 2000.

\bibitem[NV02]{NV}
Dmitri Nikshych and Leonid Vainerman.
\newblock Finite quantum groupoids and their applications.
\newblock In {\em New directions in {H}opf algebras}, volume~43 of {\em Math.
  Sci. Res. Inst. Publ.}, pages 211--262. Cambridge Univ. Press, Cambridge,
  2002.

\bibitem[Rin98]{Ringel}
Claus~Michael Ringel.
\newblock The preprojective algebra of a quiver.
\newblock In {\em Algebras and modules, {II} ({G}eiranger, 1996)}, volume~24 of
  {\em CMS Conf. Proc.}, pages 467--480. Amer. Math. Soc., Providence, RI,
  1998.

\bibitem[RR19]{RR2019}
Manuel~L Reyes and Daniel Rogalski.
\newblock Growth of graded twisted {C}alabi-{Y}au algebras.
\newblock {\em Journal of Algebra}, 539:201--259, 2019.

\bibitem[RR22]{RR2018}
Manuel~L. Reyes and Daniel Rogalski.
\newblock Graded twisted {C}alabi-{Y}au algebras are generalized
  {A}rtin-{S}chelter regular.
\newblock {\em Nagoya Math. J.}, 245:100--153, 2022.

\bibitem[RRT02]{RR-Taft}
Suemi Rodr\'{\i}guez-Romo and Earl Taft.
\newblock Some quantum-like {H}opf algebras which remain noncommutative when
  {$q=1$}.
\newblock {\em Lett. Math. Phys.}, 61(1):41--50, 2002.

\bibitem[RRZ14]{RRZ}
Manuel Reyes, Daniel Rogalski, and James~J. Zhang.
\newblock Skew {C}alabi-{Y}au algebras and homological identities.
\newblock {\em Adv. Math.}, 264:308--354, 2014.

\bibitem[Sch98]{Schauenburg-xR}
Peter Schauenburg.
\newblock Face algebras are {$\times_R$}-bialgebras.
\newblock In {\em Rings, {H}opf algebras, and {B}rauer groups
  ({A}ntwerp/{B}russels, 1996)}, volume 197 of {\em Lecture Notes in Pure and
  Appl. Math.}, pages 275--285. Dekker, New York, 1998.

\bibitem[Sch03]{Schauenburg-wha}
Peter Schauenburg.
\newblock Weak {H}opf algebras and quantum groupoids.
\newblock In {\em Noncommutative geometry and quantum groups ({W}arsaw, 2001)},
  volume~61 of {\em Banach Center Publ.}, pages 171--188. Polish Acad. Sci.
  Inst. Math., Warsaw, 2003.

\bibitem[Tak90]{Takeuchi}
Mitsuhiro Takeuchi.
\newblock A two-parameter quantization of {${\rm GL}(n)$} (summary).
\newblock {\em Proc. Japan Acad. Ser. A Math. Sci.}, 66(5):112--114, 1990.

\bibitem[Wei19]{Weis}
Stephan Weispfenning.
\newblock Properties of the fixed ring of a preprojective algebra.
\newblock {\em Journal of Algebra}, 517:276--319, 2019.

\bibitem[WW16]{WaltonWang}
Chelsea Walton and Xingting Wang.
\newblock On quantum groups associated to non-{N}oetherian regular algebras of
  dimension 2.
\newblock {\em Math. Z.}, 284(1-2):543--574, 2016.

\bibitem[WWW22]{WWW}
Chelsea Walton, Elizabeth Wicks, and Robert Won.
\newblock Algebraic structures in comodule categories over weak bialgebras.
\newblock {\em Comm. Algebra}, 50(7):2877--2910, 2022.

\end{thebibliography}

\end{document}